\documentclass[a4paper,11pt]{article}
\usepackage{amsmath, amsthm, amssymb}
\usepackage{graphicx}

\newcommand{\R}{{\cal{R}}}

\newcommand{\eps}{{\varepsilon}}
\newcommand{\pmX}{{X^{\pm 1}}}
\newcommand{\ws}{\mathcal{W}(X)}
\newcommand{\lp}[1]{\left\|{#1}\right\|}
\newcommand{\br}[2]{\left[{#1}\,;\,{#2}\right]_{ar}}

\DeclareMathOperator{\LLS}{LLS}
\newcommand{\lls}[2]{\LLS_{#1}(#2)}
\newcommand{\len}[1]{\left\vert{#1}\right\vert}

{\catcode`\|=\active
  \gdef\Set#1{\left\{\:{\mathcode`\|"8000\let|\SetVert #1}\:\right\}}}
{\catcode`\|=\active
  \gdef\Pres#1{\left\langle\:{\mathcode`\|"8000\let|\SetVert #1}\:\right\rangle}}
\def\SetVert{\egroup\;\middle|\;\bgroup}

\newtheorem{theorem}{Theorem}
\newtheorem{corollary}[theorem]{Corollary}
\newtheorem{lemma}[theorem]{Lemma}

\newtheorem{proposition}[theorem]{Proposition}

\newtheorem{remark}[theorem]{Remark}
\newtheorem{notation}[theorem]{Notation}
\newtheorem{observation}[theorem]{Observation}
\newtheorem{conjecture}[theorem]{Conjecture}

\theoremstyle{definition}
\newtheorem{definition}[theorem]{Definition}

\title{On Shephard Groups with Large Triangles}
\author{Uri Weiss \\ uriw@tx.technion.ac.il}

\begin{document}

\maketitle

\begin{abstract}
Shephard groups are common extensions of Artin and Coxeter groups. They appear, for example, in algebraic study of manifolds. An infinite family of Shephard groups which are not Artin or Coxeter groups is considered. Using techniques form small cancellation theory we show that the groups in this family are bi-automatic.
\end{abstract}

\section{Introduction}

Let $a$ and $b$ be two letters in some alphabet and let $n$ be an element of $\mathbb{N}\cup\Set{\infty}$. We denote the alternating word in $a$ and $b$ of length $n$ that starts with $a$ by $\br{a,b}{n}$, i.e., $\br{a,b}{n} = ababa \cdots$ (for $n=\infty$ we take the empty word). An \emph{Artin} relation is a relation of the type $\br{a,b}{n}= \br{b,a}{n}$. A Shephard group is a group which has presentation which consists of Artin relations and relations of the type $a^p=1$ where $a$ is a generator and $p$ is a natural number. More formally, if $X=\Set{x_1,x_2,\ldots,x_n}$ is a (finite) generating set and $\mathcal{C}=(m_{ij})_{i,j=1}^{n}$ is a symmetric matrix with $m_{ij}\in\Set{2,3,4,\ldots}\cup\Set{\infty}$ then the Shephard group $G(X,\mathcal{C})$ is the group with the following presentation:
\[
\Pres{ X | \begin{array}{l}
  \br{x_i, x_j}{m_{ij}} = \br{x_j, x_i}{m_{ji}}, \quad 1\leq i<j\leq n \\
  x_i^{m_{ii}}=1, \quad i=1,\ldots,n
\end{array}
}
\]
A Shephard group with $m_{ii}=2$ for all $1\leq i\leq n$ is a Coxeter group. A Shephard group where $m_{ii}=\infty$ for all $1\leq i\leq n$ is an Artin group. Thus, from the combinatorial group theoretical point of view, Shephard groups are common generalizations of Artin Groups and Coxeter groups. Finite Shephard groups appear as groups of reflections of the complex plane \cite{ST54}. Infinite Shephard groups have appeared in context of algebraic geometry; for example see \cite{KM98}.

We consider the problem of bi-automaticity \cite{EPS92} of Shephard groups. Coxeter groups are known to be automatic \cite{BH93} and in many cases bi-automatic \cite{AS83,NR03}. The question of automaticity for Artin groups is still open (but several sub-classes are known to be automatic and bi-automatic; see for example \cite{Cha92, Pei96}). The following conjecture is an extension of a similar conjecture for Artin and Coxeter groups. In this work we show that it holds for an infinite family of Shephard groups which are not Artin or Coxeter groups.

\begin{conjecture}
Shephard groups are bi-automatic.
\end{conjecture}

Each Shephard group $G=G(X,\mathcal{C})$ has a naturally associated graph which we term `Shephard graph'. The Shephard graph of $G(X,\mathcal{C})$ will be denote by $\Gamma(X,\mathcal{C})$. This is a labelled graph with labels both on the edges and on the vertices. The definition follows: the vertices are the generators of $G$, namely, $x_1,\ldots, x_n$, that are labeled by $m_{11},\ldots, m_{nn}$; there is an edge from $x_i$ to $x_j$ if and only if $m_{ij}\neq\infty$ and then this edge is labeled by $m_{ij}$. Notice that in contrast to the Coxeter graph, in which there is an edge $(x_i,x_j)$ if $m_{ij}\geq3$, here we have an edge if $m_{ij}<\infty$.

Our focus in this work is on a sub-class of Shephard groups in which we can apply tools from small cancellation theory. This class is distinguished by the following assumption on the structure of the Shephard graph.

\begin{definition}[Large Triangles Condition]
A Shephard group $G(X,\mathcal{C})$ has the \emph{large triangle condition} property if any closed simple path in the Shephard graph $\Gamma(X,\mathcal{C})$ of length three (i.e., a `triangle') contains no edge labeled by two. Such groups will be referred to as ``\emph{Large Triangles Shephard Groups}''.
\end{definition}

An \emph{edge subgroup} in a Shephard group $G(X,\mathcal{C})$ is a subgroup that is generated by two generators that are connected by an edge in the Shephard graph $\Gamma(X,\mathcal{C})$. Given two generators $x_i$ and $x_j$ that are connected in the Shephard graph we refer to the subgroup generated by them as the \emph{$(x_i,x_j)$-edge-subgroup}. These subgroups played an important role for Artin groups \cite{AS83}. In this work we will assume that all edge-subgroups are finite.

\begin{theorem}[Main Theorem]\label{thm:mainApplication}
A large triangles Shephard group where each edge-groups is finite is bi-automatic.
\end{theorem}

Each $(x_i,x_j)$-edge-subgroup is an homomorphic image of a group with the following presentation
\[
\Pres{x_i,x_j|x_i^{m_{ii}}=x_j^{m_{jj}}=1, \br{x_i,x_j}{m_{ij}} =
\br{x_j,x_i}{m_{ji}} }
\]
We conclude the introduction by pointing out that by classification of Shephard and Todd \cite{ST54} the group with the presentation given above is finite when either $m_{ij}\geq3$ and
\[
\frac{1}{m_{ii}}+\frac{1}{m_{jj}}+\frac{2}{m_{ij}} > 1
\]
or when $m_{ij}=2$ and both $m_{ii}$ and $m_{jj}$ are finite. This gives a sufficient (easily checked) condition in which the requirements of Theorem \ref{thm:mainApplication} are satisfied.

The rest of the paper is organized as follows. In Section \ref{sec:perliminaries} we give the basic notations and definitions. In Section \ref{sec:biAutoStruct} we describe a bi-automatic structure for the groups considered in this work. In Section \ref{vanKampenDiag} and Section \ref{sec:derDiagAdmiDiag} we introduce the notions of van Kampen diagrams, derived diagrams, and $V(6)$ diagram. In Section \ref{sec:structV6} an important property of $V(6)$ diagrams is given (Theorem \ref{thm:diagStructCC}) and we describe a construction that allows us to apply it for derived diagrams. Finally, in Section \ref{proofBiAuto} we give the proof of the main theorem. An appendix is included which contains some technical results on finite Shephard groups on two generators that are needed throughout the sections.

This work is part of the author's Ph.D. research conducted under the supervision of Professor Arye Juh\a'{a}sz.

\section{Preliminaries} \label{sec:perliminaries}

\begin{notation}\label{not:standardNot}
Let $\Pres{X|\R}$ be a finite presentation for a group $G$. Denote by $\ws$ the set of all finite words with letters in $\pmX=X\cup X^{-1}$. The elements of $\ws$ are not necessarily freely-reduced. Let $W$ and $U$ be words in $\ws$. We use the following notations:
\begin{enumerate}
 \item $\overline{W}$ denotes the element in $G$ which $W$ presents. The projection map $\pi:\ws\to G$ which sends $W$ to $\overline{W}$ is called the \emph{natural map}. We will say that $W=U$ in $G$ if $\overline{W} = \overline{U}$.


 \item $|W|$ is the length of $W$ (i.e., the number of letters in $w$).

 \item $W$ is called \emph{geodesic} in $G$ if for every $U\in\ws$ such that $\overline{W}=\overline{U}$ we have $|W|\leq|U|$.

 \item $W(n)$ is the prefix of $W$ consisting of the first $n$ letters of $W$. If $n>|W|$ then $W(n)=W$.

 \item The \emph{symmetric closure} of $\R$ is the finite subset of $\ws$ which consists of all cyclic conjugates of elements of $\R$ and their inverses.


 \item $\lp{W}$ denotes the syllable length of $W$. Namely, if $W=W_1 W_2 \cdots W_k$ where $W_j \in \mathcal{W}(x_{i_j})$ and $x_{i_j}\neq x_{i_{j+1}}$ ($j=1,2,\ldots,k-1$) then $\lp{W}=k$. Each $W_j$ is called an $x_{i_j}$-syllable of $W$. We do not assume here that $W$ nor its components $W_j$ are freely-reduced.

 \item \label{not:llsNum} Let $W=W_1 W_2 \cdots W_k \in \ws$ where $W_j \in \mathcal{W}(x_{i_j})$ and $x_{i_j}\neq x_{i_{j+1}}$. Denote by $\lls{x}{W}$ the length of the longest $x$-syllable in $W$. Namely, for $x\in X$, $\lls{x}{W} = \max\Set{\len{W_j} | W_j \in \mathcal{W}(x)}$ (we define the maximum of an empty set to be zero).

 \item \label{not:nonTriv} We say that $W\in\ws$ is equal to $1$ \emph{non-trivially} if $W=1$ in $G$ but $W$ is not freely-reducible to $1$.
\end{enumerate}
\end{notation}


\begin{definition}[Cayley Graph and the associated metric]
The Cayley graph of a group $G$ with generating set $X$ is the graph whose vertex set is $G$ and there is a directed edge from $g_1$ to $g_2$ labeled by $x\in\pmX$ if and only if $g_1x=g_2$. We denote this graph by $Cay(G,X)$. The metric $d(\cdot,\cdot)=d_X(\cdot,\cdot)$ in $Cay(G,X)$ is the standard non-directed \emph{path metric} (also called the \emph{word-metric} of $G$). Namely, $d(g_1,g_2)$ is the edge length of the shortest path from $g_1$ to $g_2$. Each word $W=x_1x_2\cdots x_k$ in $\ws$ corresponds naturally to a path in $Cay(G,X)$ whose vertices are $\overline{x_1}, \overline{x_1x_2}, \overline{x_1x_2x_3}, \ldots, \overline{W}$. For two words $W,U\in\ws$ we denote by $d(W,U)$ the distance between $\overline{W}$ to $\overline{U}$ in $Cay(G,X)$.
\end{definition}

\begin{definition}[Fellow Travellers \cite{EPS92}]
Let $k$ be a positive number. Two words $W,U\in \ws$ are called \emph{$k$-fellow-travelers} if for all $\ell\in\mathbb{N}$ we have:
\[
d_X\left(W(\ell),U(\ell)\right)\leq k
\]
Suppose $L\subseteq\ws$. $L$ has the \emph{fellow-travelers property} if there is a constant $k$ for which the following condition holds: if $W$ and $U$ are elements of $L$ and $x,y\in\pmX\cup\Set{\varepsilon}$ such that $xW=Uy$ in $G$ then $xW$ and $Uy$ are $k$-fellow-travelers.
\end{definition}

\begin{definition}[Bi-automatic structure and Bi-automatic group \cite{EPS92}]
A \emph{bi-automatic structure} of $G$ with generating set $X$ is a regular subset of $\ws$ which maps onto $G$ under the natural map and which has the fellow-traveler property. A group is \emph{bi-automatic} if it has a bi-automatic structure.
\end{definition}

The following observation will be useful for checking fellow-traveling properties.

\begin{observation} \label{obs:CheckingFeloTrvlProp}
Let $G$ be a group finitely generated by $X$ and let $W$ and $U$ be two elements of $\ws$. Suppose $W$ decomposes as $W=W_1 W_2 W_3$ and $U = W_2$ in $G$. Suppose further that $V$ is formed by replacing $W_2$ with $U$ in $W$, i.e., $V=W_1 U W_3$. In this case, $W$ and $V$ are $(|W_2|+|U|)$-fellow-travelers.
\end{observation}

\section{Bi-automatic structure for Shephard group} \label{sec:biAutoStruct}

Let $G=G(X,\mathcal{C})$ be a large triangles Shephard group with finite edge-subgroups. We next define an order ``$\prec_p$'' on $\ws$. The order is defined by attaching to each $W\in\ws$ an integer vector and then using it to compare between the different words.

\begin{definition}[Peifer vector for $G(X,\mathcal{C})$]\label{def:PeiferVector}
First we define three subsets of $\ws$. $A$ is the finite subset
\[
A=\Set{x^{\eps_1} x^{\eps_2} | x\in \pmX;\; \eps_1,\eps_2
\in \Set{-1,+1}}
\]
Suppose $m_{ij}\neq\infty$. Then, we define the set $B_{ij}^{(k)}$ by
\[
B_{ij}^{(k)} = \Set{W \in \mathcal{W}(x_i,x_j) |
\lp{W}\geq k}
\]
If $m_{ij}=\infty$ then $B_{ij}^{(k)}=\phi$ (the empty set). The set $B^{(k)}$ is the union of all the sets $B_{ij}^{(k)}$ over all $i$ and $j$. Suppose $|W|=n$. The vector $\lambda_W$ is the vector $(a_1,a_2,\ldots,a_n)$ in $\Set{0,1,2,3}^n$ where
\[
a_i = \left\{\begin{array}{ll}
  0 & \textrm{if $W(i)$ has a suffix in $A$}; \\
  1 & \textrm{if $W(i)$ has a suffix in $B^{(3)}$ but no suffix in $A$}; \\
  2 & \textrm{if $W(i)$ has a suffix in $B^{(2)}$ but no suffix in $B^{(3)}\cup A$}; \\
  3 & \textrm{otherwise.} \\
\end{array}
\right.
\]
In words, $a_i$ is zero if the prefix $W(i)$ ends with two occurrences of the same generator; $a_i$ is one if $W(i)$ ends with at least three alternating syllables of two generators which are connected in the Shephard graph; $a_i$ is two if $W(i)$ ends with exactly two syllables of two generators which are connected in the Shephard graph; $a_i$ is three in any other case.
\end{definition}

For completeness we next give the definition of lexicographical order, which is used to compare between the Peifer vectors.

\begin{definition}[Lexicographical order]
Let $v_1=(a_1,\ldots,a_k)$ and $v_2=(b_1,\ldots,b_\ell)$ be two vectors with entries in $\mathbb{N}$. We say that $v_1$ precedes $v_2$ in \emph{lexicographical order} if either $k<l$ or $k=l$ and the following holds: there is an index $1\leq i \leq k$ such that for all indexes $1\leq j < i$ there is an equality $a_j=b_j$ and for the index $i$ we have $a_i<b_i$.
\end{definition}

\begin{definition}[Peifer order on $\ws$] \label{def:PeiferOrder}
Let $W,U\in\ws$. We say that $W\prec_p U$ if one of the following conditions hold:
\begin{enumerate}
 \item $|W|<|U|$.
 \item $|W|=|U|$ and either $\lambda_W=\lambda_U$ or $\lambda_W$ precedes $\lambda_U$ in lexicographical order.
\end{enumerate}
\end{definition}

We say that $W\in\ws$ is ``$\prec_p$''-minimal (read: Peifer minimal) if for every other $U\in\ws$ such that $W=U$ in $G$ we have that $W\prec_p U$. 

\begin{remark} \label{rem:freeRedReducePeiferOrd}
Suppose that $W$ is not freely-reduced and $U$ is derived from $W$ by a single free reduction (i.e., $U=U_1U_2$ and $W=U_1 x x^{-1} U_2$). In this case we have that $U = W$ in $G$ and that $|U|<|W|$ which implies that $U\prec_p W$. In addition we also have that $U$ and $W$ are $2$-fellow-travelers (Observation \ref{obs:CheckingFeloTrvlProp}).
\end{remark}

Denote by $L_G$ the set of ``$\prec_p$''-minimal elements of $\ws$. To prove the main theorem (Theorem \ref{thm:mainApplication}) we need to give a bi-automatic structure for the Shephard group $G$. The bi-automatic structure is given in next proposition.

\begin{proposition} \label{prop:LGisBiAutoStruct}
$L_G$ is a bi-automatic structure for $G$.
\end{proposition}

Clearly, Proposition \ref{prop:LGisBiAutoStruct} establishes the main theorem (Theorem \ref{thm:mainApplication}). The rest of the work is devoted for proving this proposition. For that end we need to show that $L_G$ has the three properties needed to be a bi-automatic structure (i.e., regularity, onto, and fellow-traveler). By the definition of ``$\prec_p$'' it follows that any subset of $\ws$ contains a (Peifer) minimal element. Thus, $L_G$ is not empty and is onto $G$ through the natural map. To show that $L_G$ is regular we will use Lemma 29 of \cite{Pei96}:

\begin{lemma} \label{lem:falseKFT}
\cite[Lemma 29]{Pei96}. Let $G$ be finitely generated by $X$ and let ``$\prec$" be a regular order on $\ws$ that has minimal element for every non-empty subset of $\ws$. Assume there is a positive constant $k$, such that for every word $W$ that is not ``$\prec$"-minimal there is a word $V$ with the following properties:
\begin{enumerate}
 \item $V\prec W$.
 \item $V=W$ in $G$.
 \item $V$ and $W$ are $k$-fellow-travelers.
\end{enumerate}
Then, the set of ``$\prec$"-minimal words is a regular set.
\end{lemma}

Clearly, ``$\prec_p$'' is regular (because the subsets $A$, $B_2$, and $B_3$ are regular; see \cite{Pei96} for the meaning of regular order) and has a minimal element for every non-empty subset of $\ws$. Thus, we will show in the sequel that the rest of the conditions of Lemma \ref{lem:falseKFT} hold. Having done so we will have to show that $L_G$ has the fellow-traveler property. This will be done using the following lemma from \cite{EPS92} (given here with minor adaptation):

\begin{lemma} \label{lem:husdDist}
\cite[Lemma 3.2.3]{EPS92}. Let $G$ be finitely presented by $\Pres{X|\R}$, let $W,U\in\ws$ be two geodesics, and let $x,y\in\pmX\cup\Set{\eps}$. Assume there is a positive constant $s\in\mathbb{N}$ such that for every prefix $P_1$ of $xW$ there is a prefix $P_2$ of $Uy$ and $V\in\ws$ with $\len{V}\leq s$ for which $\overline{P_1 V}=\overline{P_2}$. Assume further that the same holds when the roles of $xW$ and $Uy$ are exchanged. Then, $xW$ and $Uy$ are $(2s+1)$-fellow-travelers.
\end{lemma}

In the terminology of \cite{EPS92}, the words $xW$ and $Uy$ in Lemma \ref{lem:husdDist} are of $s$-Hausdorff distance (see \cite{EPS92}). We will show that if $W,U\in L_G$ and there are $x,y\in\pmX\cup\Set{\eps}$ such that $xW=Uy$ in $G$ then $xW$ and $Uy$ are of $s$-Hausdorff distance for some fixed $s$ (which depends only on $G$). Since the elements of $L_G$ are geodesics, this would establish that $L_G$ has the fellow-traveler property.

\section{van Kampen Diagrams} \label{vanKampenDiag}

For establishing fellow-traveler properties we need to estimate distances in the Cayley graph. For that, we use one of the basic tools of small cancellation theory, namely, van Kampen diagrams. We next give the usual definitions and notations taken mainly from \cite[Chapter V]{LS77}.

\begin{definition}
A \emph{map} $M$ is a finite planar connected 2-complex (see \cite[Chapter V]{LS77}). We use the common convention and refer to the $0$-cells, $1$-cells, and $2$-cells of $M$ as \emph{vertices}, \emph{edges}, and \emph{regions}, respectively.
\end{definition}

All maps are assumed to be connected and simply connected unless we note otherwise. Vertices of valence one or two are allowed. Regions are open subset of the plane which are homeomorphic to open disk and edges are images of open interval (since any planner diagram may be realized on the Euclidean plane using straight lines, one may safely consider the edges as being straight lines and each region as an image of an interior of a polygon). A \emph{sub-map} of $M$ is a map whose vertices, edges, and regions are also regions, edges, and regions of $M$, respectively. Each edge $e$ of $M$ is equipped with an orientation (i.e., a specific choice of beginning and end) and we denote by $e^{-1}$ the same edge but with reversed orientation; $i(e)$ will denote the beginning vertex of $e$ and $t(e)$ will denote the ending vertex of $e$. A \emph{path} in a map $M$ is a sequence of edges $e_1,e_2,\ldots,e_k$ such that $t(e_i)=i(e_{i+1})$ for all $i=1,2,\ldots,k-1$. Given a path $\delta=e_1 e_2 \cdots e_k$ we define $i(\delta)$ to be $i(e_1)$ and $t(\delta)$ to be $t(e_k)$; we denote by $\delta^{-1}$ the reversed path, namely, $e_k^{-1} e_{k-1}^{-1} \cdots e_1^{-1}$. For a path $\delta$ we denote by $|\delta|$ the length of $\delta$ which is the number of edges it contains. A \emph{spike} is a vertex of valence one in $M$. A \emph{boundary path} of a map $M$ is a path that is contained in $\partial M$; a \emph{boundary cycle} is a closed simple boundary path. The term \emph{neighbors}, when referred to two regions, means that the intersection of the regions' boundaries contains an edge; specifically, if the intersection contains only vertices, or is empty, then the two regions are not neighbors. \emph{Boundary regions} are regions with outer boundary, i.e., the intersection of their boundary and the map's boundary contains at least one edge. \emph{Boundary edges} and \emph{boundary vertices} are edges and vertices on the boundary of the map. It will be important later to distinguish between boundary regions and \emph{proper boundary region} which we define next:

\begin{definition}[Proper boundary region]
Let $M$ be a map and $D$ a boundary region. We say that $D$ is a \emph{proper boundary region} if the \emph{closure} of $M \setminus (D \cup \partial D)$ is connected. For a proper boundary region $D$ we have that $\partial M \cap \partial D$ is a \emph{disjoint} union of connected component $C_1, C_2, \ldots C_k$, $k\geq 1$, only one which contains an edge. The \emph{outer boundary} of $D$ is the component $C_i$ which contains an edge. See Figure \ref{fig:propBndRgn}.
\end{definition}

\begin{figure}[ht]
\centering
\includegraphics[totalheight=0.18\textheight]{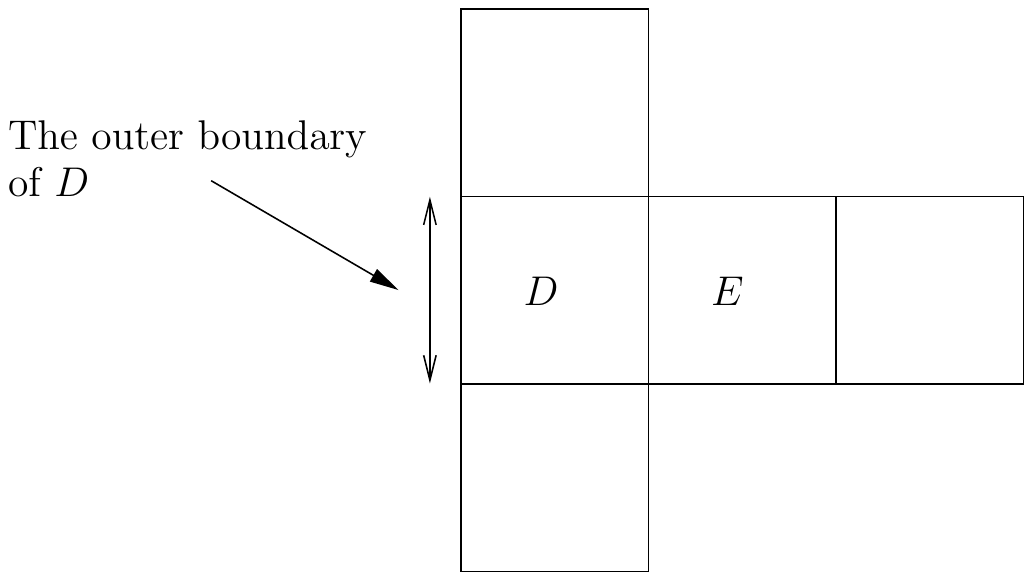}
\caption{Region $D$ is a proper boundary region while region $E$
is a boundary region but not a proper one.} \label{fig:propBndRgn}
\end{figure}

We next turn to van Kampen diagrams which are maps with specific choice of labeling on the their edges.

\begin{definition}
Let $M$ be a map. A \emph{labeling function} on $M$ with labels in group $F$ is a function $\Phi$ defined on the set of edges of $M$ and which sends each edge to an element of $F$ such that $\Phi(e^{-1})=\Phi(e)^{-1}$. We naturally extends $\Phi$ to paths in the $1$-skeleton of $M$ by sending a path $e_1 e_2 \cdots e_k$ to $\Phi(e_1) \Phi(e_2) \cdots \Phi(e_k)$. Given a finite presentation $\Pres{X|\R}$, an \emph{$\R$-diagram} is a map $M$ together with a labeling function $\Phi$ such that $\Phi(e)\in\ws$ and the images of boundary cycles of regions are elements of the symmetric closure of $\R$. $\R$-diagrams will be also referred to as \emph{van Kapmen diagram} and sometimes as just diagram if the set $\R$ is known.
\end{definition}

Suppose we are given a group $G$ with presentation $\Pres{X|\R}$. van Kampen theorem \cite[Chapter V]{LS77} states that a word $W\in\ws$ presents the identity of $G$ if and only if there is an $\R$-diagram with a boundary cycle labeled by $W$. A van Kampen diagram with boundary label $W$ is called \emph{minimal} if it has the minimal number of regions out of all the $\R$-diagrams with boundary cycle labeled by $W$.

A map $M$ with boundary cycle $\delta\mu^{-1}$ is \emph{$(\delta,\mu)$-thin} if every region $D$ has at most two neighbors and its boundary $\partial D$ intersects both $\delta$ and $\mu$. A map is \emph{thin} if it is $(\delta,\mu)$-thin for some boundary paths $\delta$ and $\mu$ and a diagram is thin if its underling map is thin. See Figure \ref{fig:thinEquaDiag} for an illustration of a thin map. If two elements $W$ and $U$ present the same element in a group $G$ then $WU^{-1}=1$ in $G$ and so by the van Kampen theorem there is a van Kampen diagram $M$ with boundary label $WU^{-1}$. We call such diagram \emph{equality diagram for $W$ and $U$}. If in addition the diagram is $(\delta,\mu)$-thin and the labels of $\delta$ and $\mu$ are $W$ and $U$, respectively, then we say that $M$ is a \emph{thin equality diagram for $W$ and $U$}.

\begin{figure}[ht]
\centering
\includegraphics[totalheight=0.17\textheight]{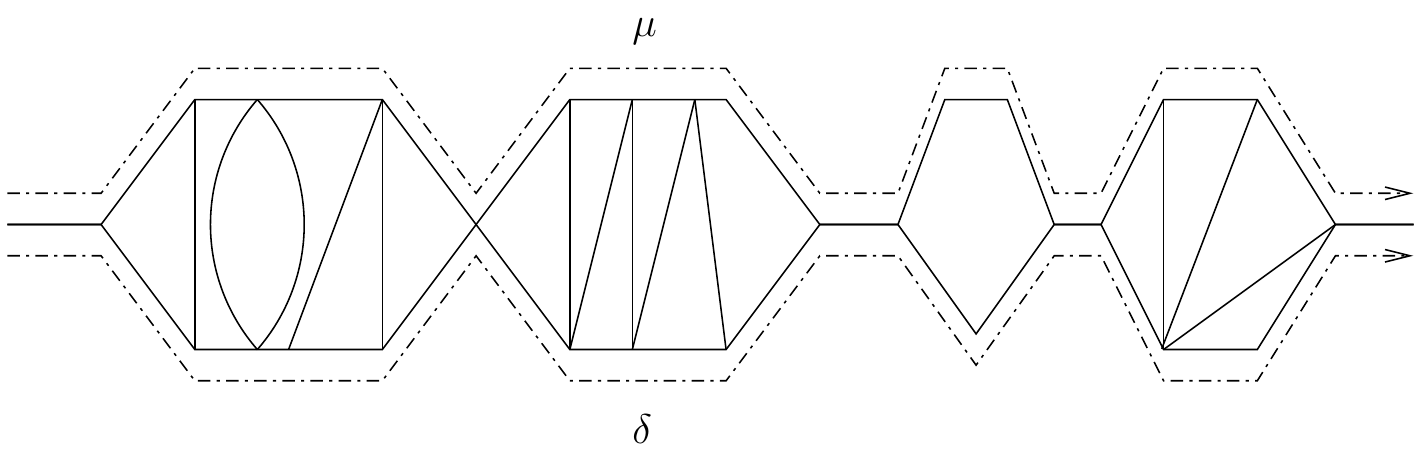}
\caption{An illustration of thin map}\label{fig:thinEquaDiag}
\end{figure}


$V(6)$ diagrams are a special type of diagrams which will be used extensively in this work. The triangle condition induces a $V(6)$ structure on admissible derived diagrams of $G$ (see Definition \ref{def:derivedDiagram} and the sequel).

\begin{definition}[$V(6)$ diagram] \label{def:v6Diag}
A map $M$ is a called a \emph{$V(6)$ map} if the following holds: suppose $D$ is an \emph{inner region such that $\partial D$ is a simple closed path} then:
\begin{enumerate}
 \item $D$ has at least four neighbors.
 \item If the boundary of $D$ contains a vertex of
 valence three then $D$ has at least six neighbors.
\end{enumerate}
A van Kampen diagram is a $V(6)$ diagram if its underlying map is
a $V(6)$ map.
\end{definition}

$V(6)$ diagrams are a special type of a larger class of diagrams known as $W(6)$ diagrams, which are described in \cite{Juh89}. The next theorem summarizes some of the properties of such diagrams which are used later. In section \ref{sec:structV6} we give further important properties of $V(6)$ maps.

\begin{theorem}[Theorems 2.3 in \cite{Juh89}]\label{thm:w6Struct}
Let $M$ be a $W(6)$ diagram. The following holds:
\begin{enumerate}
 \item Every boundary cycle of a region is a simple path.
 \item The intersection of boundary paths of every two regions is connected or empty.
\end{enumerate}
\end{theorem}

\section{Derived diagrams and Admissible diagrams} \label{sec:derDiagAdmiDiag}

Shephard groups are not small cancellation groups. We therefore apply the (now standard) procedure of derived diagrams, which allows us to use geometric arguments.

\begin{definition}\label{def:derivedDiagram}
Let $M$ be a van Kampen diagram and $M'=\Set{\Delta_1,\ldots,\Delta_k}$ be a finite collection of sub-diagrams of $M$. We say that $M'$ is a \emph{derived diagram of $M$} if the following conditions hold:
\begin{enumerate}
 \item Each $\Delta_i$ is connected and simply connected. Moreover, $\Delta_i$ has connected, non-empty interior and no spikes (vertices of valence one). In other words, the interior of $\Delta_i$ is homomorphic to an open disk and the closure of the interior of $\Delta_i$ is exactly $\Delta_i$.
 \item Any two $\Delta_i$ and $\Delta_j$ have disjoint interior.
 \item The closure of the interior of $M$ is the union of the closure of the elements in $M'$.
\end{enumerate}
The elements of $M'$ are called the \emph{derived regions of $M'$}. To distinguish the derived region from the regions of $M$ we will sometimes refer to the regions of $M$ as \emph{regular} regions. $M'$ has a natural diagram structure where its regions are the derived regions and the boundary is the boundary of $M$. This structure is used (often implicitly) in the sequel.
\end{definition}

Given a diagram $M$ we can form the \emph{trivial derived diagram} of $M$ whose derived regions are exactly the regular regions of $M$. Other derived diagrams of $M$ may be formed by joining together several regions of $M$ into one derived region. In that case we say that the derived region is formed by a \emph{gluing} of regular regions. For convenience, in the sequel we will talk about a derived diagram without explicitly mentioning the underlying diagram it was derived from. Specifically, we will discuss derived diagrams which have a given boundary label without mentioning the underling diagram. This will only be done for convenience and it should be clear that any derived diagram $M'$ is derived from some van Kampen diagram $M$.

Suppose $G=G(X,\mathcal{C})$ is a Shephard group with relations $\R=\R(X,\mathcal{C})$ (as given in the introduction) and suppose $M$ is a van Kampen $\R$-diagram with connected interior. There are two types of regions in $M$ according to the two types of relations in $\R$. The first type of regions are the regions for which their boundary labels contain two different generators. The second type is the region whose boundary label is a power of single generator. We next describe a gluing scheme (that is rigorously given below) in which neighboring regions of the first type and which are labeled by the same two generators are glued together. This technique will generate derived diagrams in which all its derived regions are in themselves diagrams over one of the edge sub-groups.

\begin{definition}[$ij$-type derived region]
Let $M$ be a van Kampen diagram over a Shephard group $G(X,\mathcal{C})$ and let $M'$ be a derived diagram of $M$. Suppose $\Delta$ is a derived region in $M'$. We call $\Delta$ an \emph{$ij$-type region} ($i\neq j$) if it contains only regular regions that their label is in $\mathcal{W}(x_i,x_j)$. $\Delta$ is called \emph{proper $ij$-type} if one of the following holds:
\begin{enumerate}
 \item $\Delta$ contains at least one regular region $D$ that its boundary label contains occurrences of both $x_i$ and $x_j$ (or their inverses).
 \item The interior of $\Delta$ is one of the connected components of the interior of $M$.
\end{enumerate}
$\Delta$ is called \emph{proper} if it is proper $ij$-type for some $i$
and $j$.
\end{definition}

\begin{definition} \label{def:admisDerDiag}
Let $M$ be a van Kampen diagram over a Shephard group $G(X,M)$ and let $M'$ be a derived diagram over $M$.
$M'$ is called \emph{pre-admissible} if for any derived region $\Delta$, we have that $\Delta$ is an $ij$-type region for some $i$ and $j$. We say that a pre-admissible derived diagram $M'$ is \emph{admissible} if the following conditions hold:
\begin{enumerate}
 \item All derived regions in $M'$ are proper.
 \item If $\Delta_1$ and $\Delta_2$ are neighbors and $\Delta_1$ is an $ij$-type region then $\Delta_2$ is not an $ij$-type region.
\end{enumerate}
\end{definition}

Suppose we are given a van Kampen diagram $M$ and a derived diagram $M'$ of $M$ that is pre-admissible (e.g., the trivial derived diagram has this property). Our next goal is to show that there is a gluing process of the regions in $M'$ that results in an admissible derived diagram.

\begin{proposition}\label{prop:gluingProccess}
Let $M$ be a van Kampen diagram over a Shephard group $G$ and let $M'=\Set{\Delta_1,\ldots,\Delta_n}$ be a derived diagram of $M$ which is pre-admissible. Then, either $M'$ is admissible or we can form a new derived diagram $M''$ by gluing two regions in $M'$ such that $M''$ is pre-admissible. More formally, we can find two adjacent derived regions $\Delta_r$ and $\Delta_s$ such that by replacing them with $\Delta_{r,s}$, a sub-diagram of $M$ formed by gluing $\Delta_r$ and $\Delta_s$, we get the derived diagram
\[
M'' = ( M' \setminus \Set{\Delta_r,\Delta_s} ) \cup \Set{\Delta_{r,s}}
\]
which is pre-admissible and $|M''|<|M'|$.
\end{proposition}

Before we prove Proposition \ref{prop:gluingProccess}, let us state
the following simple topological lemma.

\begin{lemma} \label{lem:gluingTwoRgns}
Suppose $D_1$ and $D_2$ are two regions in a map $M$ and $e$ is an edge in $\partial D_1\cap\partial D_2$. By gluing $D_1$ and $D_2$ along $e$ (or by removing $e$ from the map) we get a valid map.
\end{lemma}

See figure \ref{fig:glueTwoReg} for the illustration of Lemma \ref{lem:gluingTwoRgns}. Recall that we can regard $M$ as a subset of the Euclidean plane and we can regard $D_1$ and $D_2$ as images of interiors of polygons. The essential property of a region in a diagram is that its interior is connected and simply connected (namely, a homeomorphic image of the interior of an open disk). Therefore, we have to show that the result has a simply connected interior. Let $P_1$ and $P_2$ be two closed polygons which are copied on the closure of $D_1$ and $D_2$, respectively, and their interiors are homeomorphically copied on $D_1$ and $D_2$, respectively. Let $f_i:P_i\to M$ be the maps that copy $P_i$ to the closure of $D_i$, $i=1,2$. The edges of $P_i$ are copied onto the edges of $\partial D_i$ ($i=1,2$), edge by edge. Take the pre-images $e_1$ and $e_2$ of $e$ in $P_1$ and $P_2$, respectively, and construct $P$ by gluing $P_1$ and $P_2$ along it and taking the closure. Let $D=D_1\cup\Set{e}\cup D_2$. $P$ is a polygon that can be copied on the closure of $D$ such that its interior is homoeomorphically copied on $D$ and thus $D$ has simply connected interior. We shall call below to the process just described as "gluing".

\begin{figure}[ht]
\centering
\includegraphics[totalheight=0.25\textheight]{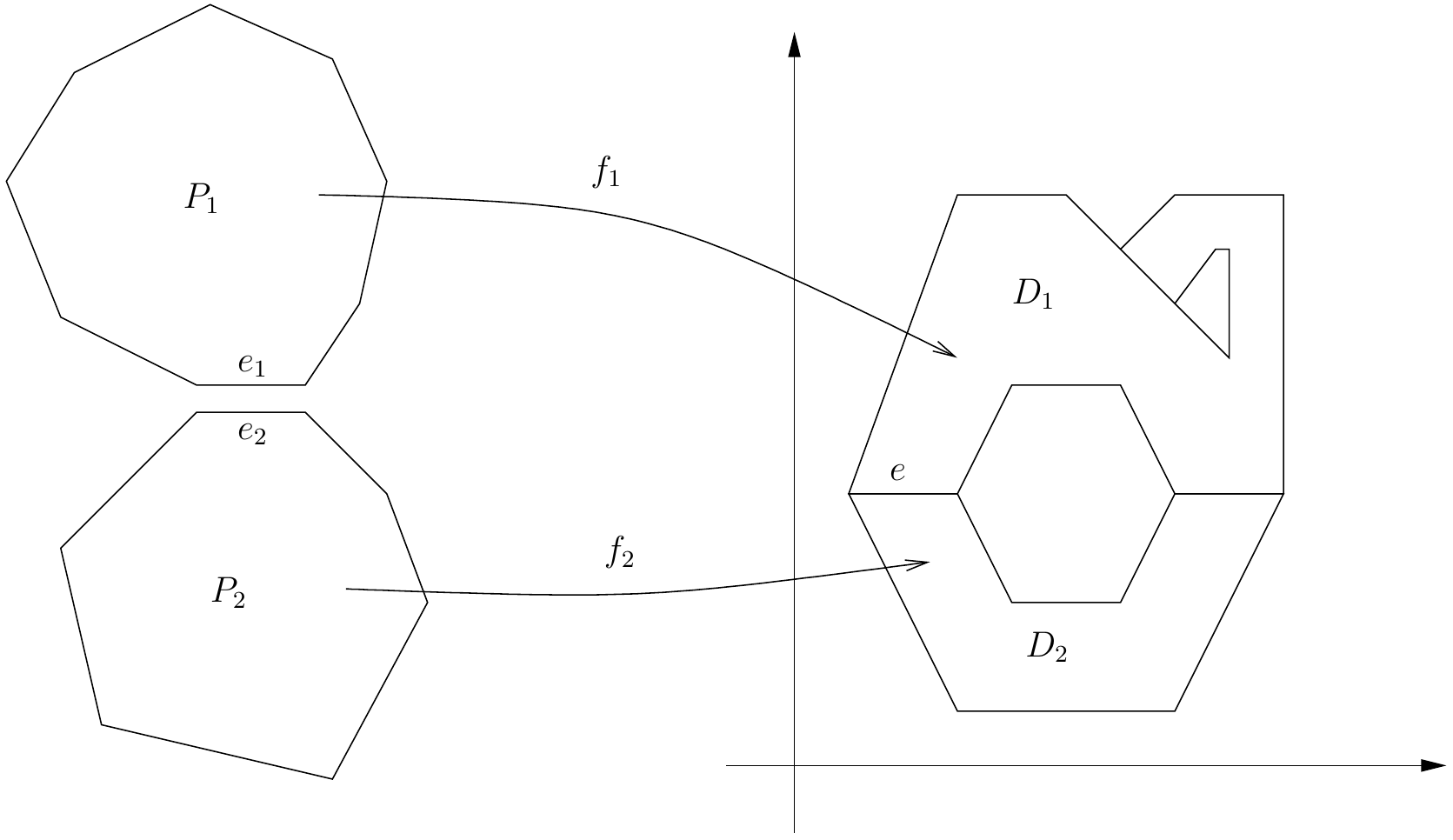}
\caption{Gluing two regions} \label{fig:glueTwoReg}
\end{figure}

\begin{proof}[Proof of Proposition \ref{prop:gluingProccess}]
Let $M$ and $M'$ be the diagram and derived diagram given in the lemma. We are done if $M'$ is admissible. Hence, assume that $M'$ is pre-admissible but not admissible. Since $M'$ is not admissible one of the two conditions of Definition \ref{def:admisDerDiag} do not hold. We next show how to construct $M''$ in each of the cases.
\begin{enumerate}
 \item Suppose $\Delta$ is a derived region in $M'$ that is not proper. Then, (1) $\Delta$ contains only regions of $M$ with boundary labeled by $x_i^{m_{ii}}$ for some $i$, and (2) its interior is not one of the connected components of the interior of $M$. Hence, there is a derived region $\Lambda$ in $M'$ that is a neighbor of $\Delta$. Since $\Lambda$ is a neighbor of $\Delta$ its outer label contains $x_i$ or its inverse. Consequently, we can form $M''$ by gluing $\Delta$ and $\Lambda$ and the resulting derived diagram is pre-admissible.
 \item Suppose $\Delta_1$ and $\Delta_2$ are two neighboring derived regions that are both $ij$-type. We can form $M''$ by gluing the two regions.
\end{enumerate}
\end{proof}

\begin{corollary} \label{cor:admisDerDiagExst}
If $M$ is a van Kampen diagram over a Shephard group then there is an admissible derived diagram for $M$.
\end{corollary}
\begin{proof}
As stated before, the trivial derived diagram of $M$ is pre-admissible. Thus, by Proposition \ref{prop:gluingProccess} there is a series $N_1,N_2,\ldots,N_k$ of derived diagrams of $M$ which are all pre-admissible and such that the number of regions in $N_{i+1}$ is one less the number of derived regions in $N_i$. If the series does not terminate with an admissible diagram we can extend it further until it terminate with an admissible derived diagram of $M$.
\end{proof}

We thus established the existence of admissible derived diagrams. We next analyze their properties in the context of large triangles Shephard groups. We start with few simple observations.

\begin{observation} \label{obs:derRgnInAdmisDerDiag}
Let $M'$ be an admissible derived diagram. Then, the following hold:
\begin{enumerate}
 \item \label{obs:derRgnInAdmisDerDiag:1} Suppose $\Delta_1$ and $\Delta_2$ are two neighboring derived regions in $M'$ and $\rho$ is a connected component of $\partial\Delta_1 \cap \partial\Delta_2$. Suppose further that $\Delta_1$ is an $ij$-type and $\Delta_2$ is a $k\ell$-type region. Then, the label of $\rho$ consists of \emph{one and only one} of the generators in $\Set{x_i,x_j,x_k,x_\ell}$. The reason for this is that if $\rho$ is labeled by two different generators then we would get that $\Delta_1$ and $\Delta_2$ are both $ij$-type (because $\rho$ would determine that) but that is impossible because $M'$ is admissible.
 \item \label{obs:derRgnInAdmisDerDiag:2} Suppose $\Delta$ is an inner derived region in $M'$ such that the boundary cycle of $\Delta$ is a simple closed path. If $W$ is the boundary label of $\Delta$ then $\lp{W}$ bounds from below the number of neighbors of $\Delta$. This follows from the previous observation since each syllable of $W$ contributes at least one neighbor to $\Delta$.
 \item  \label{obs:derRgnInAdmisDerDiag:3} If we have two neighboring derived region, $\Delta_1$ and $\Delta_2$, in $M'$ which are $ij$-type and $k\ell$-type respectively, then we must have that one (and only one) of the following equalities holds:
     \[
     i = k,\; i=\ell,\; j=k,\; j=\ell
     \]
     This follows from the first observation since $\Delta_1$ and $\Delta_2$ have the same generator appearing in the label of their boundaries.
 \item \label{obs:derRgnInAdmisDerDiag:4} Suppose $\Delta_1$, $\Delta_2$ and $\Delta_3$ are derived regions such that each two of them are neighbors. Then, there is a triangle in the Shephard graph between the three generators that appear on the boundaries of the three derived regions. This follows from previous observation as follows. Since $\Delta_1$, $\Delta_2$ and $\Delta_3$ are neighbors they are $ij$-type, $jk$-type, and $ki$-type, respectively, for some indices $i$, $j$, and $k$. Then, the generators $x_i$ and $x_j$, the generators $x_j$ and $x_k$, and the generators $x_k$ and $x_i$ are connected in the Shephard graph. Hence, $x_i$, $x_j$, and $x_k$ form a triangle in the Shephard graph.
\end{enumerate}
\end{observation}

Let $x_i$ and $x_j$ be two generators which are connected in the Shephard graph. We denote by $E_{ij}$ the following group:
\[
\Pres{x_i,x_j|x_i^{m_{ii}}=x_j^{m_{jj}}=1, \br{x_i,x_j}{m_{ij}} =
\br{x_j,x_i}{m_{ji}} }
\]
Note that if $\Delta$ is a derived region which is an $ij$-type then $\Delta$ is a (regular) van Kampen diagram over $E_{ij}$ (because the regular regions in $\Delta$ have boundary label that is one the relations of $E_{ij}$). Consequently, the boundary label $W$ of $\Delta$ is equal to $1$ in $E_{ij}$.


We next introduce an assumption on the edge-subgroups of the Shephard group. The reader may want to recall the definition of of being equal to $1$ non-trivially in a group (see Notation \ref{not:standardNot} - part \ref{not:nonTriv}).

\begin{definition}[Appel-Schupp Syllable Length Condition] \label{def:ASSyllLenCond}
Let $G(X,M)$ be a Shephard group and let $x_i$ and $x_j$ be two generators that are connected in the Shephard graph. We say that the \emph{Appel-Schupp Syllable Length condition} holds if $\lp{W}\geq 2m_{ij}$ for any $W\in\mathcal{W}(x_i,x_j)$ such that $W=1$ in $E_{ij}$ non-trivially. 
\end{definition}

This condition holds for Artin groups of large type \cite[Lemma 6]{AS83}. It is also true for large triangle Shephard groups with finite edge-subgroups; the proof of this assertion is left to the appendix. We next give an important geometric property of admissible derived diagrams over large triangles Shephard groups. The proposition assumes some additional properties which will be easily satisfied using the construction that follows.

\begin{proposition}\label{prop:admissV6}
Let $M$ be a van Kampen diagram over a large triangles Shephard group $G$ in which Appel-Schupp syllable condition holds. Suppose $M'$ is an admissible derived diagram over $M$ such that if $\Delta$ is an $ij$-type derived region in $M'$ then the boundary label $W$ of $\Delta$ is not trivially equal to $1$ in $G$. 
Then, $M'$ is a $V(6)$ diagram.
\end{proposition}
\begin{proof}
Suppose $\Delta$ is an inner $ij$-type derived region of $M'$ such that $\partial \Delta$ is a simple path (recall that it is enough to consider regions with simple boundary paths). Suppose $W$ is a boundary label of $\Delta$. By the Observation \ref{obs:derRgnInAdmisDerDiag} - part \ref{obs:derRgnInAdmisDerDiag:2}, $\lp{W}$ bounds from below the number of neighbors of $\Delta$ and by the assumptions of the proposition we have that $W$ equals to $1$ in $E_{ij}$ non-trivially. 
Thus, by Appel-Schupp syllable length condition we get that $\lp{W}\geq4$. Hence, $\Delta$ has at least four neighbors. If $\Delta$ has a vertex of valence three then there are two regions $\Delta'$ and $\Delta''$ that are neighbors of $\Delta$ and of each other. Hence, the three generators that appear on the boundaries of $\Delta$, $\Delta'$, and $\Delta''$ form a triangle in the Shephard graph (Observation \ref{obs:derRgnInAdmisDerDiag} - part \ref{obs:derRgnInAdmisDerDiag:4}). Thus, by the triangle condition, $m_{ij}\geq3$. Hence, once again by Appel-Schupp syllable length condition we have that $\lp{W}\geq6$ and thus $\Delta$ has at least six neighbors.
\end{proof}

Having worked with admissible derived diagrams, we next turn to deal with a more specific type of derived diagrams ($\Omega$-minimal) which have much stronger properties, specifically, the properties that are needed for Proposition \ref{prop:admissV6}. Corollary \ref{cor:OmegeDiagProp} below is the main ingredient that allows our later arguments to work.

\begin{definition}[$\Omega$-minimal derived diagrams]
Let $M$ be a diagram and $M'$ be a derived diagram over $M$. If $\Delta$ is a derived region in $M'$ we denote by $\ell(\Delta)$ the length of the boundary label of $\Delta$ and we denote by $\Omega(M')$ the following sum:
\[
\Omega(M') = \sum_{\Delta\in M'} \ell(\Delta)
\]
A pre-admissible derived diagram with given boundary label $W$ is called \emph{$\Omega$-minimal} if $\Omega(M')$ is minimal and $M'$ has maximal number of derived regions (in that order, first select the pre-admissible derived diagrams that have minimal $\Omega(M')$ from all derived diagrams with boundary label $W$ and then select among them the one with maximal number of derived regions).
\end{definition}

Suppose we are given a series $N_1,N_2,\ldots,N_k$ of derived diagrams, each formed from the previous by gluing two derived regions (as in the proof of Corollary \ref{cor:admisDerDiagExst}). Then, the series $\Omega(N_1),\Omega(N_2),\ldots,\Omega(N_k)$ is strictly decreasing. To see why, suppose we glued $\Delta_1$ and $\Delta_2$ in $N_{i}$ over the path $\rho$ to form $N_{i+1}$. Then, the label of $\rho$ which was counted twice in $\Omega(N_i)$ is not counted at all in $\Omega(N_{i+1})$, and the rest of the sum is left untouched. This together with Proposition \ref{prop:gluingProccess} proves the following lemma:

\begin{lemma}
$\Omega$-minimal derived diagrams are admissible diagrams.
\end{lemma}

Next, we show that the derived regions in an $\Omega$-minimal derived diagram are of special type.

\begin{proposition} \label{prop:RegInOmegaMin}
Suppose $M$ is a van Kampen diagram and suppose $M'$ is an $\Omega$-minimal derived diagram over $M$. Let $\Delta$ be an $ij$-type derived region in $M'$, let $W\in\mathcal{W}(x_i,x_j)$ be a boundary label of $\Delta$, and let $\Lambda$ be a (regular) van Kampen diagram of $W$ over $E_{ij}$. Then, $W$ is freely-reduced and $\Lambda$ has connected interior.
\end{proposition}
\begin{proof}
If $W$ is not freely-reduced then we can construct a van Kampen diagram $\Lambda$ with boundary label $W$ and which has spikes (vertices of valence one). We will therefore show that $\Lambda$ has no spikes and has connected interior. Assume by contradiction that $\Lambda$ has no connected interior or it has spikes. We can do the following `surgery' on $M'$: cut out $\Delta$ and replace it with $\Lambda$. Note that in this process we also change the underling diagram $M$ (because we may change the set of regions in $M$) but the boundary label is left unchanged. However, for brevity, we also denote the new diagram by $M$. Since $\Lambda$ may have non-connected interior and/or may have spikes it is no longer a valid derived region. Thus, let $\Lambda_1,\ldots,\Lambda_k$ be the closures of the connected components of the interior of $\Lambda$. Instead of $\Delta$ we now have $k$ derived regions, $k\geq1$. The edges of $\Lambda$ which are not in the union $\Lambda_1 \cup \cdots \cup \Lambda_k$ will not be part of the boundaries of the new derived regions. Denote the new derived diagram by $M''$. Note that $\Omega(M')\geq\Omega(M'')$ because $\ell(\Delta) = |W| \geq \ell(\Lambda_1)+\cdots+\ell(\Lambda_k)$. If indeed, $k>1$ (i.e., $\Lambda$ had non-connected interior) then we now have strictly more derived regions in contradiction to $M'$ being $\Omega$-minimal diagram. Thus, we can assume that $\Lambda$ has connected interior. If $\Lambda$ had a spike then we once again get a contradiction to the $\Omega$-minimal condition since then $\Omega(M')>\Omega(M'')$ (because the edge emanating from the spike is not counted in the sum of $\Omega(M'')$).
\end{proof}

\begin{corollary} \label{cor:OmegeDiagProp}
Suppose $\Delta$ is a derived region in an $\Omega$-minimal derived diagram and let $W$ be the boundary label of $\Delta$. Then, $W$ is freely-reduced and if $U$ is a (proper) sub-word of $W$ then $U\neq1$ in the group $G$.
\end{corollary}
\begin{proof}
$W$ is freely-reduced by Proposition \ref{prop:RegInOmegaMin}. If $U$ is a sub-word of $W$ and $U=1$ in the group then after cyclic conjugation we can decompose $W$ into $UV$ such that both $U$ and $V$ are equal to $1$ in the group. Let $\Lambda_1$ and $\Lambda_2$ be two diagrams with boundary labels $U$ and $V$, respectively. By attaching $\Lambda_1$ and $\Lambda_2$ we can form a diagram $\Lambda$ that has non-connected interior which has a boundary label $W$. This once again contradicts Proposition \ref{prop:RegInOmegaMin}.
\end{proof}

To recap, suppose $G$ is a large triangle Shephard group with finite edge-subgroups and $M'$ is an $\Omega$-minimal derived diagram over a diagram $M$. Then, $M'$ is admissible. Also, if $\Delta$ is an $ij$-type derived region in $M'$ with boundary label $W$ then $W$ is freely-reduced (i.e., $W$ is not equal to $1$ non-trivially). \emph{Thus, $M'$ is a $V(6)$ diagram}. Notice that with regard to the above notation, $\lls{x_i}{W} < m_{ii}$ and $\lls{x_j}{W} < m_{jj}$ because $W$ cannot contain a sub-word of the form $x_i^{m_{ii}}$ or $x_j^{m_{jj}}$ which are equal to $1$ in $G$. As we shall see later, the fact that $W$ has no sub-word that is equal to $1$ in $G$ implies that $W$ has a bounded length (a bound depending only on $G$). This will be important to our argument later on.


\section{Structure theorem for $V(6)$ diagrams} \label{sec:structV6}

Before we turn to prove Proposition \ref{prop:LGisBiAutoStruct} we make a small detour to give a useful property of proper $V(6)$ maps (see the definition below). Since we mainly work with derived diagrams we also describe afterward how to use this property in our context.

\begin{definition} \label{def:propV6}
A $V(6)$ map is \emph{proper} if the following additional conditions holds:
\begin{enumerate}
 \item There are no inner vertices of valence two.
 \item If $D$ is a boundary region of $M$ then:
 \begin{enumerate}
  \item The boundary $\partial D$ of $D$ contains at least four edges.
  \item If $\partial D$ contains an inner vertex of valence three then $\partial D$ contains at least six edges.
 \end{enumerate}
\end{enumerate}
\end{definition}

\begin{definition}[Cut Corner] \label{def:CC}
Let $M$ be a proper $V(6)$ van Kampen map. Let $\rho=e_1 \cdots e_k$ a boundary path in $M$. Suppose $D$ is a \emph{proper} boundary region of $M$ and suppose $\rho_D=\partial D \cap \rho$ is the outer boundary of $D$ (i.e., the connected component of $\partial D \cap \partial M$ that contains an edge). Denote by $\ell$ the lowest index such that $i(e_\ell)$ is in $\rho_D$. Denote by $\delta_D$ the rest of the boundary path of $D$ starting at $i(e_\ell)$ (the complement of $\rho_D$). If $i(e_\ell)$ is of valence three then there is a neighbor $E$ of $D$ which its boundary contains $i(e_\ell)$. For this region $E$, let $\rho_E$ be the connected component of $\partial E \cap \rho$ which contains $i(e_\ell)$.

We say that $D$ is a \emph{cut corner contained in $\mu$} if one of the following conditions hold:
\begin{enumerate}
 \item[T1.] $|\delta_D|<|\rho_D|$.
 \item[T2.] $|\delta_D|=|\rho_D|=2$, $\ell>0$, and $i(e_\ell)$ is of valence three.
 \item[T3.] $|\delta_D|=|\rho_D|=3$, $\ell>0$, $i(e_\ell)$ is of valence three, and $E$ has less than six edges.
 \item[T4.] $|\delta_D|=|\rho_D|=3$, $\ell>1$, $i(e_\ell)$ is of valence three, and $|\rho_E|\geq2$.
\end{enumerate}
See Figure \ref{fig:CC} for illustrations of cases 1 through 4.
\begin{figure}[ht]
\centering
\includegraphics[totalheight=0.25\textheight]{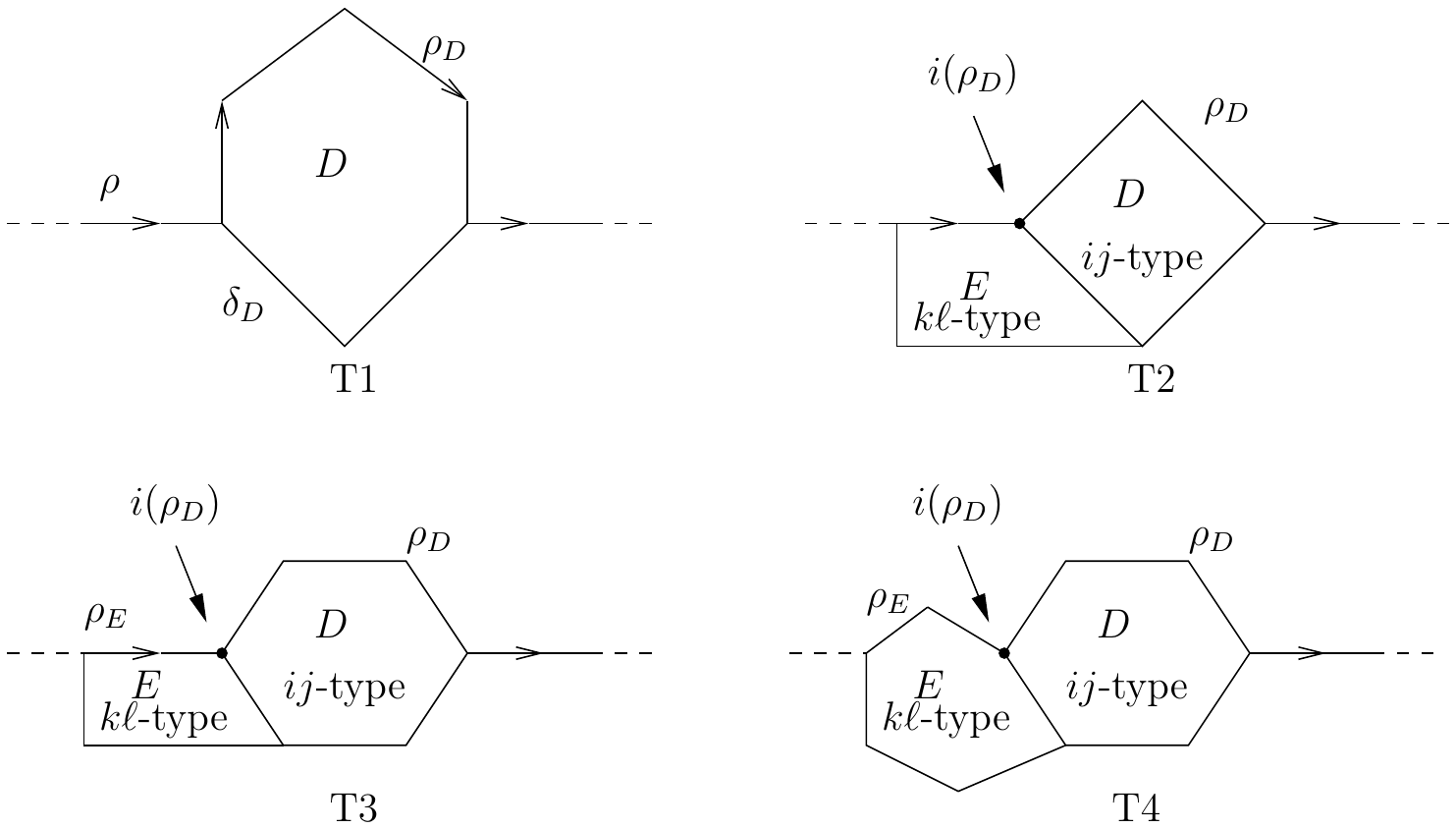}
\caption{Different types of cut corners} \label{fig:CC}
\end{figure}
\end{definition}

\begin{remark} \label{rem:OutEdgAtlstAsInEdgInCC}
Note that if $D$ is a cut corner with $r$ inner edges and $s$ boundary edges then $r\leq s$ and $r\neq s$ only in the T1 case. Consequently, if $r=s$ then $D$ is of type T2, T3, or T4 in which $r\leq 3$. Moreover, if $r=s$ then the vertex $i(e_\ell)$ lays on the boundaries of exactly two regions (since it is of valence three).
\end{remark}

Next theorem link between cut corners, thin maps, and proper $V(6)$ maps. It is one of the main ingredients in the proof of Proposition \ref{prop:LGisBiAutoStruct}.

\begin{theorem}[Theorem 13 of \cite{Wei07}] \label{thm:diagStructCC}
Suppose $M$ is a proper $V(6)$ van Kampen map. Suppose further that the boundary of $M$ has a decomposition $\sigma_1 \delta \sigma_2^{-1} \mu^{-1}$ such that $|\sigma_1|\leq 1$ and $|\sigma_2|\leq 1$. If there is \emph{no} cut corner contained in either $\mu$ or $\delta$ then $M$ is a $(\sigma_1 \delta, \mu \sigma_2)$-thin map.
\end{theorem}

Let $G=G(X,\mathcal{C})$ be a large triangles Shephard group where its edge-subgroups are finite.  Suppose we are given an $\Omega$-minimal derived diagram $M'$ over a diagram $M$. $M'$ is a $V(6)$ diagram (Proposition \ref{prop:admissV6}) however for Theorem \ref{thm:diagStructCC} we need a \emph{proper} $V(6)$ map. To solve this we need to define a set of edges for $M'$ that would turn it into a proper $V(6)$ diagram. First,  we take as inner edges for $M'$ all the paths in $M$ that are on the boundaries of two different regions (i.e., paths of the form $\partial \Delta \cap \partial \Lambda$ for two derived regions $\Delta$ and $\Lambda$). We do so since there may not be inner vertices of valence two. For boundary edges of $M'$ we do the following. Let $\Delta$ be a boundary region in $M'$, let $\rho$ is a connected component of $\partial \Delta \cap \partial M$, and let $V$ be the label of $\rho$. Then, $V=V_1 V_2 \cdots V_k$ such that $V_i$ are the syllables of $V$. We may decompose $\rho$ into $\rho = \rho_1 \rho_2 \cdots \rho_k$ such that $\rho_i$ is labeled by $V_i$ (possibly, by introducing vertices of valence two along $\rho$). This specific choice of edges in $M'$ will be called  \emph{syllable-induced edges}. Note that by Appel-Schupp Syllable Length Condition this specific choice of edges induces a proper $V(6)$ structure on $M'$. A syllable-induced set of edges define a set of \emph{syllable-induced vertices} which are the initial and terminal vertices of the syllable-induced edges.

\begin{figure}[ht]
\centering
\includegraphics[totalheight=0.18\textheight]{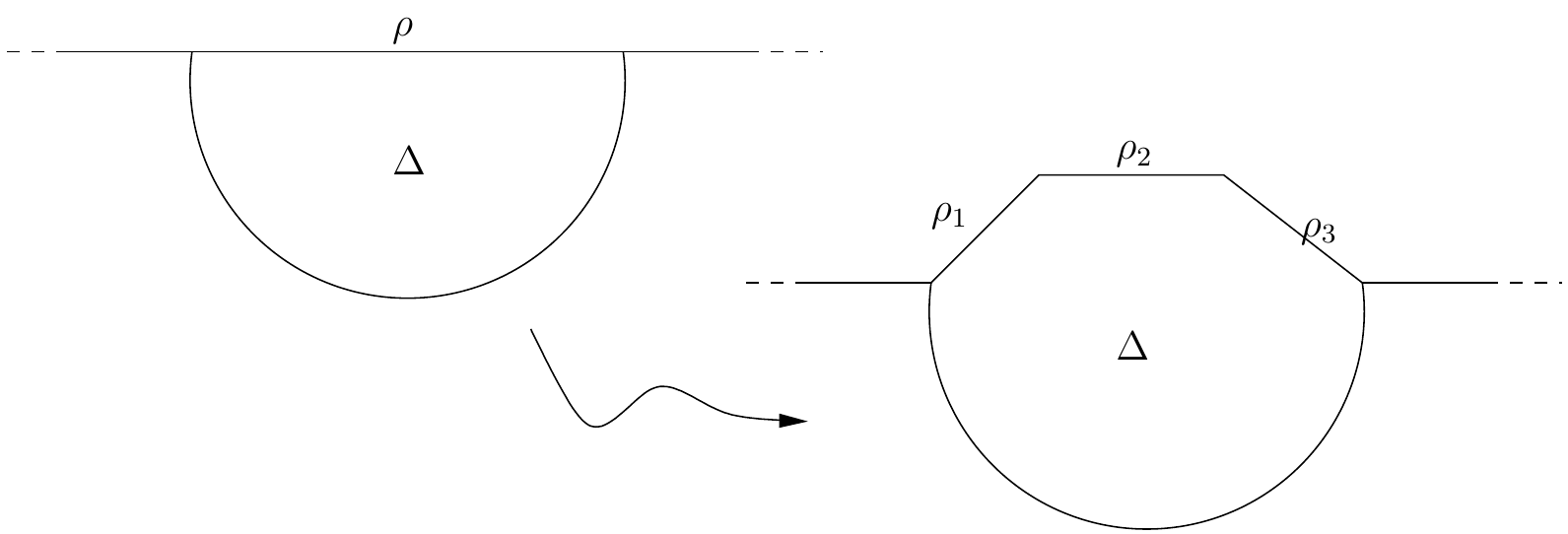}
\caption{Making a diagram a proper $V(6)$ diagram} \label{fig:properV6}
\end{figure}

As it turns out, using syllable-induced edges for considering a $V(6)$ proper diagrams is not enough for our needs later on. Thus, we refine the above construction as follows: suppose $\Delta$ is a proper boundary derived region in a derived diagram $M'$ (as above) and suppose $\rho$ is the boundary path of $\Delta$. In this case we do a \emph{minimal} split of $\rho$ into $\rho = \rho_1 \rho_2 \cdots \rho_k$ (we minimize on the number of elements in the split, i.e., $k$; the exact locations of the splits is unimportant) such that $\Delta$ has at least four edges and if $\partial\Delta$ contains an inner vertex of valence three then $\Delta$ has at least six edges. Namely, we make sure that conditions of proper $V(6)$ diagram hold but we take as few edges as possible. We will refer to this construction as \emph{minimal set of syllable-induced edges}. The only difference between the construction of minimal set of syllable-induced edges and the definition above (of syllable-induced set of edges) is that the number of boundary edges on the boundary of a \emph{proper} boundary regions was reduced.

\begin{remark} \label{rem:MinSylIndEdgSylLenBoundNumEdgs}
Suppose we consider a minimal set of syllable-induced edges on an admissible derived diagram $M'$ and $\Delta$ is a boundary region. Suppose further that $s$ is the number of boundary edges in $\rho$ which is a connected path of $\partial \Delta \cap \partial M$ and $W$ is the label of $\rho$. Then, $s\leq\lp{W}$. The reason is that $s$ may be at most the number of syllables in $W$ by the construction (it may be less, due to the minimality, but not more).
\end{remark}

The reason for the special choice we made regarding boundary edges is the following lemma.

\begin{lemma} \label{lem:SixEdgeCC}
Let $M'$ be an $\Omega$-minimal derived diagram and consider a minimal set of syllable-induced edges. Let $\Delta$ be a cut corner with $r$ inner edges and $s$ boundary edges. Then, $r\leq 3$ and if $r=3$ then $s=3$.
\end{lemma}
\begin{proof}
Suppose by contradiction that $\Delta$ is a cut corner of type T1 with $r\geq3$. In cut corner of type T1 we have $r<s$ so the number of edges in $\partial \Delta$ is at least $7$ (because $r\geq3$ and $s\geq4$). This situation is clearly not minimal in the number of edges (we can take less edges in the outer boundary of $\Delta$ while still having a proper $V(6)$ diagram). Thus, such condition will contradict the minimality of the set of syllable-induced edges. If $\Delta$ is a cut corner of other type (type T2, type T3, or type T4) then $r\leq3$ so we get that in all cases $r\leq 3$. To show that if $r=3$ then $s=3$ note that: it follows if $\Delta$ is a cut corner of type T3 or of type T4; it follows trivially for type T2 (for which $r=s=2$); it follows for type T1 since then $r<3$ by the contradiction we got above.
\end{proof}

\section{Proof of Bi-Automaticity} \label{proofBiAuto}

After the preparatory section above, in this section we finally turn to prove Theorem \ref{thm:mainApplication} by proving Proposition \ref{prop:LGisBiAutoStruct}. Let $G=G(X,\mathcal{C})$ be a large triangles Shephard group where its edge-subgroups are finite. Recall that for Proposition \ref{prop:LGisBiAutoStruct} we need to show that the language $L_G$ is a bi-automatic structure for $G$ (see the proposition for the definition of $L_G$). The group $G$ and its presentation is fixed throughout this section. If $m_{ij}\neq\infty$ then the edge-subgroup that is generated by $x_i$ and $x_j$ is finite by our assumptions. We thus introduce the following notation: \emph{$\kappa(G)$ is the smallest upper bound on the sizes of finite edge-subgroups in $G$}. Namely, if $m_{ij} \neq \infty$ then the size of the edge-group that is generated by $x_i$ and $x_j$ is bounded from above by $\kappa(G)$.

\begin{lemma}\label{lem:onFiniteGrps}
Suppose $H$ is a finite group that is finitely generated by $Y$. Let $W=y_1 y_2 \cdots y_n$ be an element of $(Y^{\pm 1})^*$ where $n > |H|$. Then, $W$ has a proper sub-word that is equals to $1$ in $H$.
\end{lemma}
\begin{proof}
By the pigeonhole principle the series $y_1,\, y_1 y_2,\, \ldots,\, y_1y_2 \cdots y_n$ contains two elements which are equal in $H$.
\end{proof}

Using Lemma \ref{lem:onFiniteGrps} we can now formulate a strong bound on derived regions.

\begin{corollary}
If $M'$ is an $\Omega$-minimal derived diagram over $G$ and $\Delta$ is a derived region in $M'$ with boundary label $U$ then $|U|\leq\kappa(G)$.
\end{corollary}
\begin{proof}
By Corollary \ref{cor:OmegeDiagProp} we have that $U$ does not contain proper sub-words that are equal to $1$ in $G$. Hence, it follows from Lemma \ref{lem:onFiniteGrps} that the length of $U$ is at most $\kappa(G)$.
\end{proof}

Before we arrived to the heart of the proof of Proposition \ref{prop:LGisBiAutoStruct}, we need several technical results. The first two are properties of finite Shephard groups on two generators; the proofs of which will be given in the appendix. The reader may want to recall the definition of $E_{ij}$ groups (see the discussion before Definition \ref{def:ASSyllLenCond}) and the definition of $\lls{x}{W}$ (see Notation \ref{not:standardNot} - part \ref{not:llsNum}).

\begin{lemma} \label{lem:TechPropOfTwoGenGrp}
Let $1 \leq i,j \leq n$ and let $U,V\in\mathcal{W}(x_i,x_j)$ such that the following properties hold: $UV=1$ in $E_{ij}$ non-trivially and also $\lls{x_i}{U}\leq\frac{1}{2}m_{ii}$ and $\lls{x_j}{U}\leq\frac{1}{2}m_{jj}$. Then,
\begin{enumerate}
 \item If $\lp{U} \leq m_{ij}$ then $|U|\leq |V|$.
 \item If $\lp{U} < m_{ij}$ then $|U| < |V|$.
\end{enumerate}
\end{lemma}
\begin{proof}
See the appendix.
\end{proof}

\begin{lemma} \label{lem:ThreeSylbGeoSpecialCase}
Consider $U$ and $V$ as in Lemma \ref{lem:TechPropOfTwoGenGrp}. We make the following assumptions:
\begin{enumerate}
 \item $UV$ is freely-reduced and cyclically reduced (as written).
 \item $\lls{x_i}{U}\leq\frac{1}{2}m_{ii}$ and $\lls{x_j}{U}\leq\frac{1}{2}m_{jj}$.
 \item $\lls{x_i}{V}\leq\frac{1}{2}m_{ii}$ and $\lls{x_j}{V}\leq\frac{1}{2}m_{jj}$.
 \item If $U$ ends with $x_i^{p_1}$ and $V$ starts with $x_i^{p_2}$ then $|2p_1 + p_2| < m_{ii}$.
 \item Previous part holds if we replace the roles of $U$ and $V$ and if we replace $x_i$ with $x_j$ (and thus $m_{ii}$ with $m_{jj}$).
 \item $\lp{U}=m_{ij}=3\leq\lp{V}$.
 \item $|U|=|V|$.
\end{enumerate}
Then,
\begin{enumerate}
 \item $U$ has the form $U=a^{\eps_1} U' a^{\eps_2}$ where $a\in\Set{x_i,x_j}$ and $\eps_1,\eps_2 \in \Set{-1,1}$.
 \item $V$ has the form $V=a^{\eps_1} V' a^{\eps_2}$ where $a\in\Set{x_i,x_j}$ and $\eps_1,\eps_2 \in \Set{-1,1}$.
 \item If $U$ starts with $x_i^{\pm1}$ then $V$ starts with $x_j^{\pm1}$ or vice versa.
\end{enumerate}
\end{lemma}
\begin{proof}
See the appendix.
\end{proof}

Next lemmas deal with properties of $\Omega$-minimal derived diagrams in anticipation of the proposition that follows. The first two show why we would be able to assume the conditions of Lemma \ref{lem:TechPropOfTwoGenGrp} (see the remark that follows). The last one analyze a boundary label around a vertex which has valance three in the derived diagram.

\begin{lemma} \label{lem:noHighPowerInOmgMinDerDiag}
Let $M'$ be an $\Omega$-minimal derived diagram and let $\Delta$ and $\Lambda$ be two neighboring derived regions in $M'$. Let $\rho = \partial \Delta \cap \partial \Lambda$. By Observation \ref{obs:derRgnInAdmisDerDiag} - part \ref{obs:derRgnInAdmisDerDiag:1}, the label of $\rho$ is $x_i^p$ for some generator $x_i$. Then, $|p|\leq\frac{1}{2}m_{ii}$.
\end{lemma}
\begin{proof}
Notice that $x_i^p=x_i^{m_{ii}-p}$ in $G$. Hence, it follows that if $p>\frac{1}{2}m_{ii}$ then we can `fix' $\rho$ such that it would be labeled by $x_i^{m_{ii}-p}$ (See figure \ref{fig:edgeLen}) and thus reduce $\Omega(M')$. This move is impossible if we assume that $\Omega(M')$ is minimal so consequently we get that $p\leq\frac{1}{2}m_{ii}$. Similarly, if $p<-\frac{1}{2}m_{ii}$.

\begin{figure}[ht]
\centering
\includegraphics[totalheight=0.18\textheight]{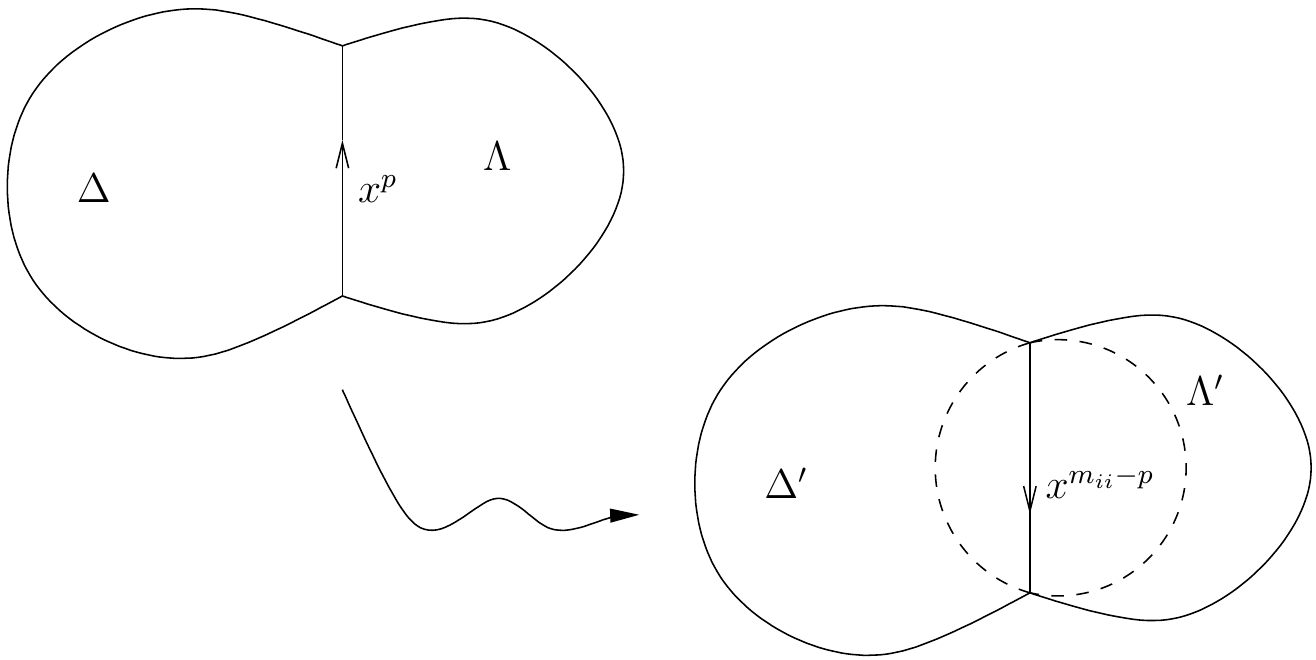}
\caption{Replacing edge labeled by $x_i^p$ with edge labeled by
$x_i^{m_{ii}-p}$} \label{fig:edgeLen}
\end{figure}
\end{proof}

\begin{lemma} \label{lem:noHighPowerInOmgMinDerDiagBoundCase}
Let $M'$ be an $\Omega$-minimal derived diagram and let $\Delta$ and $\Lambda$ be two neighboring derived regions in $M'$. Let $\rho=\rho_1 v \rho_2$ be a path in $\partial \Delta$ such that $\rho_1 = \partial \Delta \cap \partial \Lambda$, $\rho_2$ is a boundary path, and $v$ is a vertex only on the boundaries of $\Delta$ and $\Lambda$. Assume that the label of $\rho_1$ is $x_i^{p_1}$ and the label of $\rho_2$ is $x_i^{p_2}$ (for some generator $x_i$). Then, $|2p_1 + p_2|\leq m_{ii}$.
\end{lemma}
\begin{proof}
Assume by contradiction that $2p_1 + p_2 > m_{ii}$. By the properties of $\Omega$-minimal diagram we have that $p_1 + p_2 < m_{ii}$, thus we can choose $p_3$ such that $p_1 + p_2 + p_3 = m_{ii}$ ($p_1$, $p_2$, and $p_3$ are positive integers). Suppose we `fix' the diagram as described in Figure \ref{fig:edgeLenBoundCase}. Namely, we replace $\rho_1$ with a path $\rho_2 \rho_3$ such that $\rho_3$ is labeled by $x_i^{p_3}$ and we replace $\rho$ with $\rho_3$. Before the change, $\Omega(M')$ counted $p_1$ twice and after the change it instead counts $p_3$ twice (where the rest of the sum is left unchanged). We claim that this reduces $\Omega(M')$ in contradiction to the minimality. Indeed:
\[
p_3 = m_{ij} - p_1 - p_2 < 2p_1 + p_2 - p_1 - p_2 = p_1
\]

\begin{figure}[ht]
\centering
\includegraphics[totalheight=0.18\textheight]{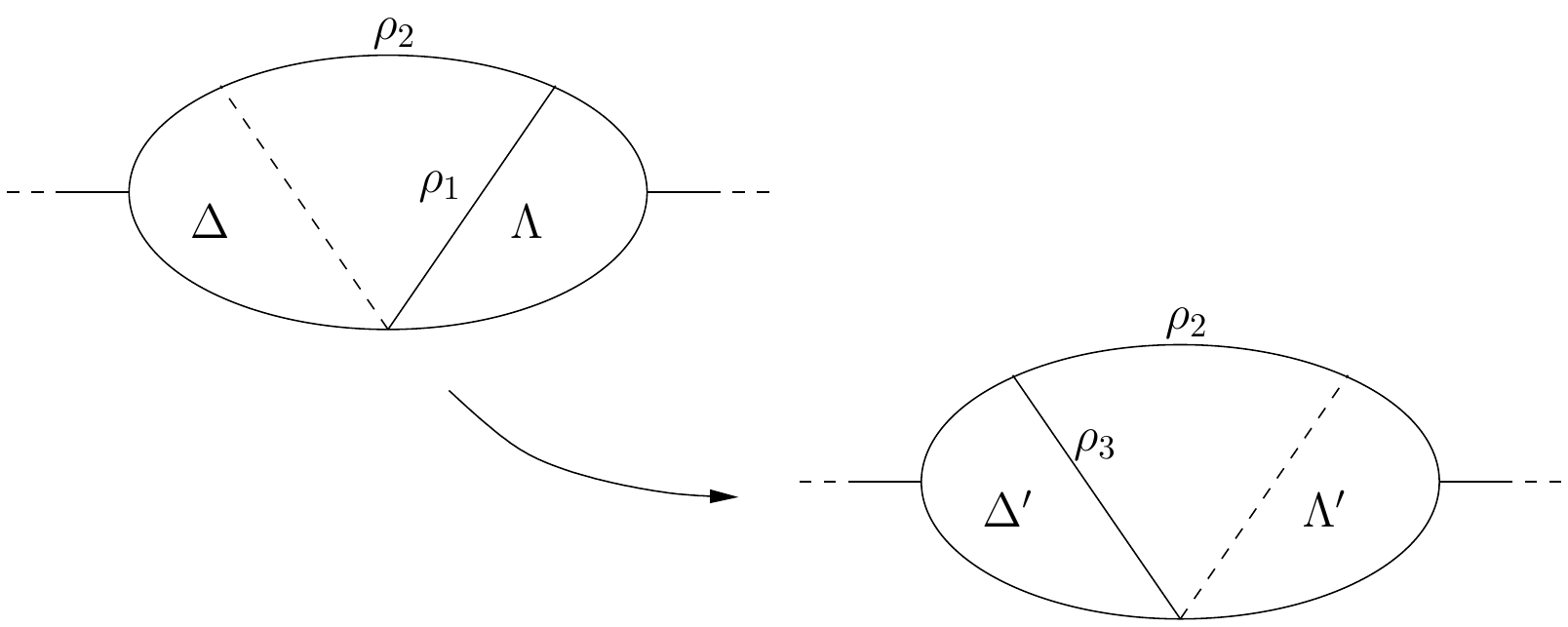}
\caption{Edge labeled with large power near the boundary} \label{fig:edgeLenBoundCase}
\end{figure}
\end{proof}

\begin{remark} \label{rem:SpecCaseOfMijIsThree}
Let $M'$ be an $\Omega$-minimal derived diagram and let $\Delta$ be an $ij$-type cut corner in $M'$. We next make a series of assumptions on $\Delta$ and on its boundary label. These assumptions are satisfied later on at one point during the proof of Proposition \ref{prop:CCtoReduce} and we would want to use Lemma \ref{lem:ThreeSylbGeoSpecialCase} there. Thus, we remark here that indeed all the condition of this Lemma are satisfied under the assumptions below.

The assumptions follow. $\Delta$ has exactly three neighbors and $m_{ij}=3$. Let $\mu_o \mu_i$ be a boundary path of $\Delta$ were $\mu_o$ is the outer boundary of $\Delta$. Also, let $U$ be the label of $\mu_i$ and $V$ the label of $\mu_o$. We assume that $V$ is geodesic in $E_{ij}$, that the number of edges in $\mu_o$ is no more than $\lp{V}$, and that $|U|=|V|$.

We next check that the conditions of Lemma \ref{lem:ThreeSylbGeoSpecialCase} for $U$ and $V$ are satisfied under these assumptions. The first one follows from the $\Omega$-minimality. The next two follow from the fact that $V$ is geodesic and Lemma \ref{lem:noHighPowerInOmgMinDerDiag}. By Lemma \ref{lem:noHighPowerInOmgMinDerDiagBoundCase} we get the forth and fifth conditions. Since $\Delta$ has three neighbors we get that $\lp{U}=3=m_{ij}$ and since $\Delta$ is a cut corner we get that the number of edges in $\mu_o$ is at least three so $3\leq\lp{V}$. This gives us the sixth condition. The last condition is one the assumptions above.
\end{remark}

\begin{lemma} \label{lem:twoDiffGenOnValance3Vtx}
Let $M'$ be an $\Omega$-minimal derived diagram and suppose $\rho=\rho_1 v \rho_2$ is a boundary path of $M'$ which contains a vertex $v$ that is not the initial or terminal of $\rho$. Let $a$ be the last letter of the label of $\rho_1$ and $b$ be the first letter of the label of $\rho_2$. If $v$ is on the boundaries of exactly two derived regions then $a\neq b$.
\end{lemma}
\begin{proof}
Let $\Delta$ and $\Lambda$ be the two adjacent derived regions that contain $v$ in their boundaries. Suppose, $\Delta$ and $\Lambda$ are $ij$-type and $jk$ type regions, respectively. See Figure \ref{fig:sixPos}. Consider the following three edges that emanate from $v$: two boundary edges and the inner edge that is joint to the boundaries of $\Delta$ and $\Lambda$; denote these edges by $e_1$, $e_2$, and $e_3$, respectively, where $e_1^{-1}e_2$ is a sub-path of $\rho$ containing $v$. Let $a$ be the last letter of the label of $e_1^{-1}$, $b$ be the first letter of $e_2$, and $c$ the first letter of $e_3$. Since the boundary labels of $\Delta$ and $\Lambda$ are freely-reduced (Corollary \ref{cor:OmegeDiagProp}) we have the following inequalities: $a\neq c^{-1}$ and $b\neq c$. Assume by contradiction that $a=b$. Then, both the boundaries of $\Delta$ and $\Lambda$ would contain the letters $a$ and $c$ (which are different because $a\neq c$ and also $a\neq c^{-1}$) so both regions are $ij$-type regions for the same $i$ and $j$. This would contradicts the admissibility of $M'$.

\begin{figure}[ht]
\centering
\includegraphics[totalheight=0.18\textheight]{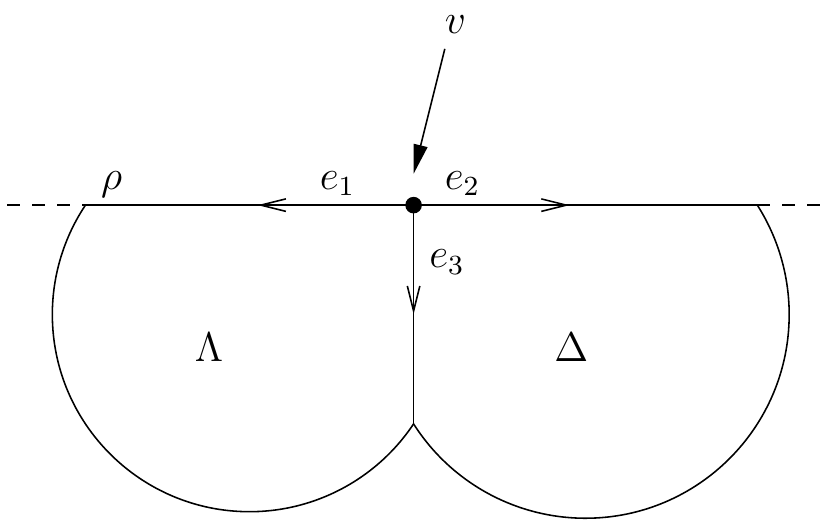}
\caption{A boundary vertex of valence three} \label{fig:sixPos}
\end{figure}

\end{proof}

And next, we give the main proposition of this section; once we establish it the rest of the proof is mostly routine. The reader may find the flow chart (Figure \ref{fig:flowChart}) that is attached after the proof helpful while reading the proof.

\begin{proposition} \label{prop:CCtoReduce}
Let $M$ be a van Kampen diagram over $G$, let $M'$ be an $\Omega$-minimal derived diagram over $M$ and consider a proper $V(6)$ structure on $M'$ through a minimal set of syllable-induced edges and vertices. Suppose $\rho$ is a boundary path of $M'$ with label $W\in\ws$. If $\rho$ contains a cut corner (a derived region in $M'$) then we can find $U\in\ws$ such that: (i) $U\prec_p W$, (ii) $U=W$ in $G$, and (iii) $U$ and $W$ are $2\kappa(G)$-fellow-travelers.
\end{proposition}
\begin{proof}
Suppose $\Delta$ an $ij$-type derived region that is a cut corner which is contained in $\rho$. Assume that $W$ decomposes as $W=W_1 W_2 W_3$ where $W_2$ is the label of the outer boundary of $\Delta$, denoted by $\mu_o$ (thus, $W_2\in\mathcal{W}(x_i,x_j)$). Let $\mu_i$ the complement of $\mu_o$ in the boundary of $\Delta$ such that $i(\mu_o)=i(\mu_i)$ (i.e., they start on the same vertex) and let $W_2'$ be the label of $\mu_i$. Finally, let $s$ be the number of outer distinguished edges in $\mu_o$ and let $r$ be the number of neighbors of $\Delta$. See Figure \ref{fig:CCInRho}. It is clear that we can assume that $W$ is freely reduced (Remark \ref{rem:freeRedReducePeiferOrd}). Since $\Delta$ is a cut corner it is a proper boundary region. Using the fact that the set of edges is a minimal set of syllable-induced edges we get that $r\leq3$, and $r\leq s$ (see Remark \ref{rem:OutEdgAtlstAsInEdgInCC} and Lemma \ref{lem:SixEdgeCC}). And, since $|W_2 (W_2')^{-1}| \leq \kappa(G)$ we have that $W$ and $W_1 W_2' W_3$ are $\kappa(G)$-fellow-travelers (Observation \ref{obs:CheckingFeloTrvlProp}) and also that $W=W_1 W_2' W_3$ in $G$. The proof proceeds in eight steps (with a small interruption for notations in the middle.)

\begin{figure}[ht]
\centering
\includegraphics[totalheight=0.18\textheight]{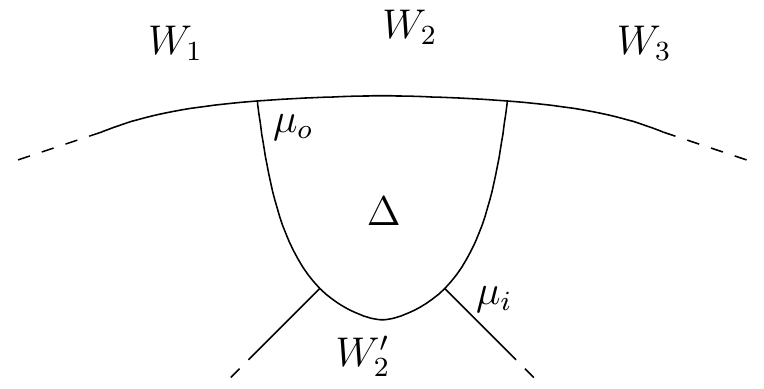}
\caption{Illustration of the situation in the proof of Proposition \ref{prop:CCtoReduce}} \label{fig:CCInRho}
\end{figure}

\begin{enumerate}

 \item \label{prop:CCtoReduce:1} In the first step we treat the case where $W_2$ is not geodesic in $E_{ij}$ (the edge-subgroup generated by $x_i$ and $x_j$). Then, there is $V\in\mathcal{W}(x_i,x_j)$ such that $|V|<|W_2|$ and $V=W_2$ in $E_{ij}$. Since $|W_2|\leq\kappa(G)$ we have that $|V|+|W_2|\leq2\kappa(G)$ and thus $U=W_1 V W_3$ has all the properties (i), (ii), and (iii), above (the first one follows because $|U|<|W|$ and third one follows due to Observation \ref{obs:CheckingFeloTrvlProp}). \emph{For the rest of the proof we will assume that $W_2$ is geodesic in $E_{ij}$}.

 \item \label{prop:CCtoReduce:2} In this step we treat the case where $m_{ij}=2$ and $\lp{W_2}\geq3$. If $\lp{W_2}\geq3$ we get that $W_2$ has a sub-word of the following form: $a^{\eps_1} b^p a^{\eps_2}$ where $\eps_1,\eps_2\in\Set{-1,1}$, $p\in \mathbb{Z} \setminus \Set{0}$,  $a,b\in\Set{x_i,x_j}$, and $a\neq b$. By our assumption in this step $m_{ij}=2$ so $ab=ba$ in $G$ and since $W_2$ is geodesic we have that $\eps_1=\eps_2$. Assume, w.l.o.g, that $\eps_1=\eps_2=1$. Let $W_2 = V_1 a b^p a V_3$ and denote by $k$ the length of $W_1V_1$. Also let $V_2 = a b^p a$, $V_2' = a^2 b^p$, and $U=W_1 V_1 V_2' V_3 W_3$. Clearly $W=U$ in $G$ (since $ab=ba$ in $E_{ij}$). We show that $U \prec_p W$. This follows from the definition of Peifer vectors (Definition \ref{def:PeiferVector}): $\lambda_{W}$ is strictly positive in the $(k+2)$ entree since, $W(k+2)$ ends with $ab$; $\lambda_{U}$ is zero in the $(k+2)$ entree since, $U(k+2)$ ends with two occurrences of $a$; the first $k+1$ entries of $\lambda_{W}$ and $\lambda_{U}$ are identical. Hence, $\lambda_{U}$ precedes the vector $\lambda_{W}$ in lexicographical order. To complete this step we note that $|V_2'|+|V_2| \leq 2 |W_2| \leq 2\kappa(G)$ so consequently $U$ and $W$ are $2\kappa(G)$ fellow-travelers. \emph{For the rest of the proof we will assume that if $m_{ij}=2$ then $\lp{W_2}\leq2$}.
 \item \label{prop:CCtoReduce:3} Recall that $r$ is the number of neighbors $\Delta$ has. In this step we show that $r \leq m_{ij}$. Since $r\leq 3$ this immediately follows if $3 \leq m_{ij}$. If $m_{ij}=2$ then by the assumption we made in step \ref{prop:CCtoReduce:2} we have that $s = \lp{W_2} \leq 2$. Thus, $\Delta$ is a cut corner of type $T2$ so $r=2$ and specifically, $r \leq m_{ij}$.

 \item \label{prop:CCtoReduce:4} In this step we show that $|W_2|=|W_2'|$ and that $r=\lp{W_2'}=m_{ij}$. Let $\mu_i=\delta_1 \delta_2 \cdots \delta_r$ be the syllable-induced edges along $\mu_i$. We have that $\lp{W_2'} \leq r$ since each of the edges along $\mu_i$ is labeled by a power of a generator (see Observation \ref{obs:derRgnInAdmisDerDiag} - part \ref{obs:derRgnInAdmisDerDiag:1}). By step \ref{prop:CCtoReduce:3} we have that $r\leq m_{ij}$ and thus $\lp{W_2'}\leq m_{ij}$. Also, we have that $W_2 (W_2')^{-1}=1$ in $E_{ij}$. Write $W_2'=x_{j_1}^{p_1} \cdots x_{j_\ell}^{p_\ell}$ where $x_{i_j}\in X$ and $x_{i_j}\neq x_{i_{j+1}}$ for $1\leq j<\ell$ (note that it is possible that $\ell<r$). By using the relations $x_i^{m_{ii}}=1$ and $x_j^{m_{jj}}=1$ in $E_{ij}$ we can write $V = x_{j_1}^{p_1'} \cdots x_{j_\ell}^{p_\ell'}$ such that $|p_j'| \leq \frac{1}{2} m_{i_j i_j}$ and $W_2'=V$ in $E_{ij}$. Now, $\lp{V}=\lp{W_2'}\leq m_{ij}$, $V W_2^{-1}=1$ in $E_{ij}$, $\lls{x_i}{V}\leq \frac{1}{2}m_{ii}$ and $\lls{x_j}{V}\leq \frac{1}{2}m_{jj}$. Hence, by Lemma \ref{lem:TechPropOfTwoGenGrp} we have that $|V|\leq|W_2|$ and since $W_2$ is geodesic we have that $|V|=|W_2|$. We claim that $|V|=|W_2'|$. Suppose otherwise, then there is an index $j$ such that $\delta_j$ and $\delta_{j+1}$ have labels that are powers of the same generator. If that happen then $\lp{V}<r\leq m_{ij}$ so Lemma \ref{lem:TechPropOfTwoGenGrp} we get that $|V|<|W_2|$ in contradiction to the equality above. To show that $r=m_{ij}$ assume by contradiction that $r<m_{ij}$ then $\lp{V}<m_{ij}$ so again by Lemma \ref{lem:TechPropOfTwoGenGrp} we'd get that $|V|<|W_2|$ in contradiction to $W_2$ being geodesic. The equality $\lp{W_2'}=m_{ij}$ follows along the same lines. 

 \item \label{prop:CCtoReduce:5} Recall that $r$ is the number of neighbors that $\Delta$ has and $s$ is the number of boundary edges $\Delta$ has. In this step we show that $(r,s) \in \Set{(2,2),(3,3)}$. From step \ref{prop:CCtoReduce:4} we have that $r=m_{ij}$ and thus $2\leq r$. If $r=2$ then the only option is that $(r,s)=(2,2)$ because of the assumption at the end of step \ref{prop:CCtoReduce:3} we have that $s\leq\lp{W_2}\leq2$ (see Remark \ref{rem:MinSylIndEdgSylLenBoundNumEdgs} for the first inequality). The next case is $r>2$. We have that $r\leq 3$ and $r\leq s$ (see Remark \ref{rem:OutEdgAtlstAsInEdgInCC}) thus we need to show that if $r=3$ then $s=3$; this follows from Lemma \ref{lem:SixEdgeCC}.
\end{enumerate}

We interrupt the sequence of steps to introduce few notations that will be used in the last steps. Let $U=W_1 W_2' W_3$ and denote the length of $|W_1|$ by $t$. Clearly, condition (ii) holds for $U$ by the construction of $U$ and condition (iii) holds by Observation \ref{obs:CheckingFeloTrvlProp} since $|W_2|+|W_2'|\leq\kappa(G)$. Thus, we need to show in the remaining cases that condition (i) holds for $U$. Let $\lambda_W=(a_1\ldots,a_t,a_{t+1},\ldots,a_n)$ and let $\lambda_U=(b_1,\ldots,b_t,b_{t+1},\ldots,b_n)$ (note that both $W$ and $U$ have the same prefix $W_1$ and thus we have that $\lambda_W$ and $\lambda_U$ agree on the first $t$ entries). Since $(r,s)\in\Set{(2,2),(3,3)}$ we get that $\Delta$ is a cut corner of type T2, type T3, or type T4 (Remark \ref{rem:OutEdgAtlstAsInEdgInCC}). Hence, the vertex $i(\mu_o)$ belongs to the boundaries of exactly two derived regions, $\Delta$ and $\Lambda$. See Figure \ref{fig:frstVertx}. Let $x_k$ be the generator such that the last letter of $W_1$ is $x_k$ or $x_k^{-1}$. Let $x_i$ be the generator such that the first letter of $W_2$ is $x_i$ or $x_i^{-1}$. Lastly, there is a single inner edges emanating from $i(\mu_o)$ that belongs to both $\partial \Delta$ and $\partial \Lambda$; let $x_\ell$ be the generator such that the first letter of the label of this edge emanating is $x_\ell$ or $x_\ell^{-1}$.  We assume that $\Lambda$ is a $k\ell$-type region. Note that since $x_\ell$ is on the boundary of $\Delta$ we have that $\ell\in\Set{i,j}$. The proof continues with the following steps:

\begin{figure}[ht]
\centering
\includegraphics[totalheight=0.18\textheight]{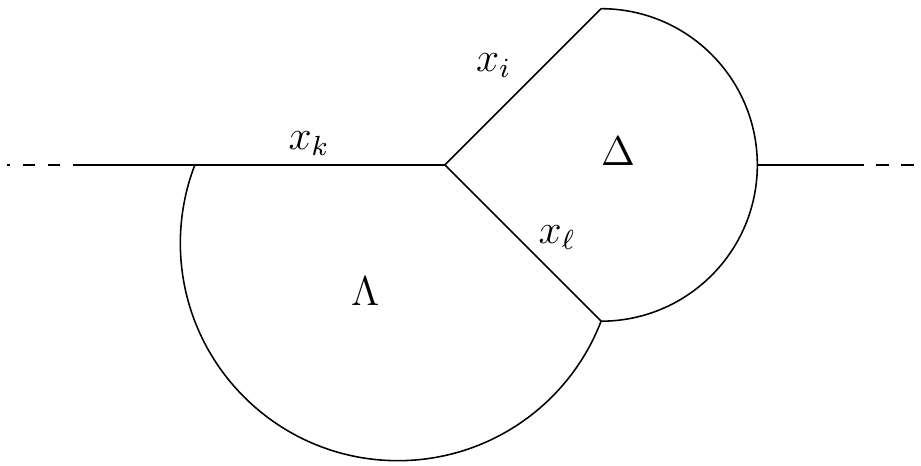}
\caption{Near the first vertex of a cut corner}
\label{fig:frstVertx}
\end{figure}

\begin{enumerate} \setcounter{enumi}{5}

 \item \label{prop:CCtoReduce:6} In this step we treat the case that $\ell=k$. If $\ell=k$ then $b_{t+1}=0$ (since, $U(t+1)$ ends with two occurrences of $x_k$) and $a_{t+1}>0$ (since, $k\neq i$ by Lemma \ref{lem:twoDiffGenOnValance3Vtx}). Therefore, an equality $k=\ell$ implies that $b_{t+1}<a_{t+1}$ and we get that $\lambda_U$ precedes $\lambda_W$ in lexicographical order and thus $U\prec_p W$. \emph{For the rest of the proof we will assume that $k\neq \ell$}. 
 
 \item \label{prop:CCtoReduce:7} In this step we treat the case that $(r,s)=(2,2)$. Here, $\Delta$ is a cut corner of type T2 (see Figure \ref{fig:CC} - case 2). We have that $m_{ij}=2$ and that the number of syllables in $\partial \Delta$ is exactly four (it is at least four by the Apple-Schupp syllable length condition and it is at most four by the assumption at the end of step \ref{prop:CCtoReduce:2}). Hence, $\ell\neq i$ so we get that $\ell=j$ and $m_{kj}\neq\infty$. Since $m_{ij}=2$ we get by the large triangles condition (i.e., no triangles in the Shephard graph with one side labeled by $2$) that $m_{ki}=\infty$. Therefore, no suffix of $W(t+1)$ is in $B^{(2)}$ or in $B^{(3)}$ (these are two of the three sets that were used to define the Peifer order - Definition \ref{def:PeiferOrder}). By Lemma \ref{lem:twoDiffGenOnValance3Vtx} we have that $k\neq i$ so no suffix of $W(t+1)$ is in $A$ and thus $a_{t+1}=3$. However, $x_k^{\epsilon_1} x_j^{\epsilon_1}$ (for $\epsilon_1,\epsilon_2\in\Set{-1,1}$) is a suffix of $U(t+1)$ and is in $B^{(2)}$ so $b_{t+1}<3$. Consequently, $U\prec_p W$. Thus, we showed that condition (i) holds for $U$ and the proposition hold in this case.

 \item In this step we treat the case that $(r,s)=(3,3)$. In this case $m_{ij}=3$ and $\Delta$ has three neighbors. All the assumptions of Lemma \ref{lem:ThreeSylbGeoSpecialCase} hold (see Remark \ref{rem:SpecCaseOfMijIsThree}) and thus the generators at the beginning of $W_2$ and $W_2'$ are different, hence $\ell\neq i$. Consequently we get that $\ell=j$. There are two sub-cases to consider, the case where $\Delta$ is a cut corner of type $T3$ and the case of type $T4$ (see Figure \ref{fig:CC} - cases 3 and 4).
     \begin{enumerate}
      \item Cut corner of type T4. We have that $\Lambda$ has at least two boundary edges and $m_{ij}=3$. Let $V$ be the label of the connected component of the outer boundary of $\Lambda$ that is adjacent to $i(\mu_o)$. We have that $\lp{V}\geq2$ (see Remark \ref{rem:MinSylIndEdgSylLenBoundNumEdgs}) and that $V$ is a suffix of $W_1$. Consequently, $W_1$ has a suffix $x_\ell^{m_1} x_k^{m_2}$ for $m_1,m_2\in\mathbb{Z}\setminus\Set{0}$. This shows that $b_{t+1}=1$ since $x_\ell^{m_1} x_k^{m_2} x_\ell^{\epsilon}\in B^{(3)}$ and that is a suffix of $U(t+1)$. On the other hand, $x_\ell^{m_1}x_k^{m_2}x_i$ is a suffix of $W(t+1)$. Hence, to show that $b_{t+1}<a_{t+1}$ it is enough to show that $W(t+1)$ has no suffix in $B^{(3)}$. This follows since $\ell=j$.
      \item Cut corner of type T3. Suppose $V$ is the label of the connected component of the outer boundary of $\Lambda$ that is adjacent to $i(\mu_o)$, which we denote by $\delta$. If $\lp{V}\geq2$ then by repeating the argument of the T4 case above we are done. Hence, we can assume that $\lp{V}=1$ and thus $\delta$ contains one edge (see Remark \ref{rem:MinSylIndEdgSylLenBoundNumEdgs}). Suppose $\Lambda$ is an $gh$-type region ($g$ and $h$ are two indexes). We show that $m_{gh}=2$. This will finish this case since we could then repeat the argument given for cut corner of type T2 (step \ref{prop:CCtoReduce:7}). So we show that $m_{gh}=2$. The assumption of cut corner of type T3 is that $\partial \Lambda$ contains four edges. It is therefore enough to show that the syllable length of $\partial \Lambda$ is at most four (hence, by Appel-Schupp syllable length condition we'd get that $2m_{gh}\leq 4$ so consequently $m_{gh}=2$). If $\Lambda$ is a proper boundary region then it has three neighbors (due to the three inner edges of $\Lambda$) so together with $V$ there are at most four syllables in the label of $\partial \Lambda$. If $\Lambda$ is not a proper boundary region then, by the construction of the minimal set of syllable-induced edges, the number of syllables in $\partial \Lambda \cap \partial M$ is equal to the number of edges in $\partial \Lambda \cap \partial M$. Thus, once again there are at most four syllables in the label of $\partial \Lambda$.
     \end{enumerate}
\end{enumerate}

\end{proof}

\begin{figure}[ht]
\centering
\includegraphics[totalheight=0.3\textheight]{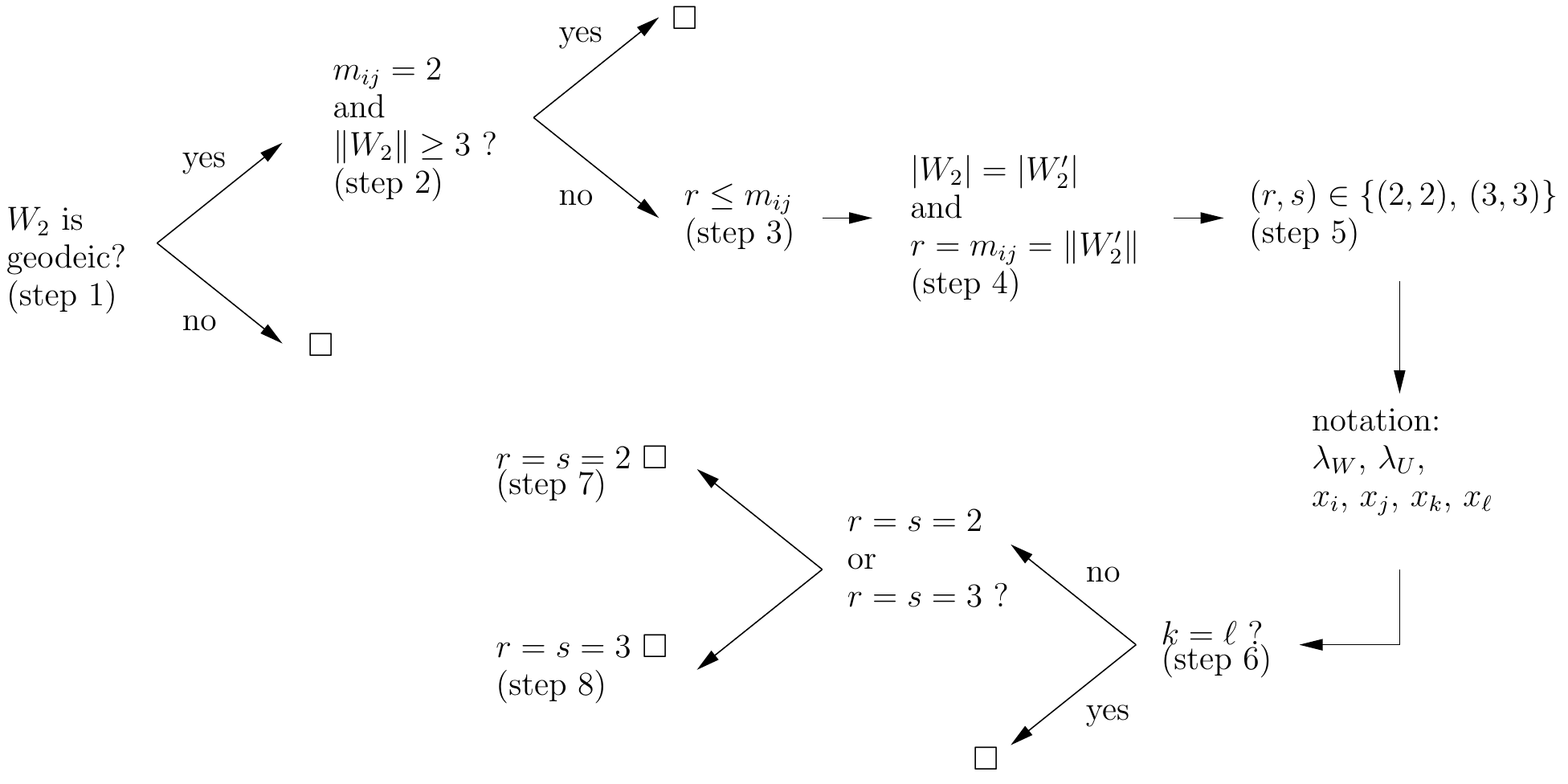}
\caption{A flow chart of the proof of Proposition \ref{prop:CCtoReduce}} \label{fig:flowChart}
\end{figure}

Using Proposition \ref{prop:CCtoReduce} we can now complete the proof of Proposition \ref{prop:LGisBiAutoStruct}. First, we show that $L_G$ is regular using Lemma \ref{lem:falseKFT}. Let $W\in\ws$ be a non-Peifer minimal. We take $U\in\ws$ such that $W=U$ in $G$ and $U$ is Peifer minimal. Let $M'$ be an $\Omega$-minimal derived diagram with boundary cycle $\rho\delta^{-1}$ such that $\rho$ labeled by $W$ and $\delta$ is labeled by $U$. Consider the proper $V(6)$ structure on $M'$ through a minimal set of syllable-induced edges. By Proposition \ref{prop:CCtoReduce} we are done if $\rho$ contains a cut corner. Thus, we can assume that $\rho$ does not contain a cut corner. By the minimality of $U$, also the path $\delta$ does not contain a cut corner. Thus, by Theorem \ref{thm:diagStructCC} we have that $M'$ is a $(\rho,\delta)$-thin diagram. We conclude with the following lemma from \cite{Pei96}:

\begin{lemma} \label{lem:thinDiagForKFT}
\cite[Lemma 21]{Pei96}. Let $M$ be a $(\rho,\delta)$-thin diagram with $\rho$ labeled by $W$ and $\delta$ labeled by $U$. Let $k$ a natural number that bounds the lengths of labels of boundaries of regions in $M$. Then, one of the following holds:
\begin{enumerate}
 \item $W$ and $U$ are $k$-fellow-travelers.
 \item There is a word $W'$ such that $W'$ and $W$ are $k$-fellow-travelers, $\len{W'}<\len{W}$, and $\overline{W'}=\overline{W}$.
 \item There is a word $U'$ such that $U'$ and $U$ are $k$-fellow-travelers, $\len{U'}<\len{U}$, and $\overline{U'}=\overline{U}$.
\end{enumerate}
\end{lemma}

By Lemma \ref{lem:thinDiagForKFT} we have that either $W$ and $U$ are $\kappa(G)$-fellow-travelers or that there is $W'$ that is shorter than $W$ and is a $\kappa(G)$-fellow-traveler of $W$. Since $\len{W'}<\len{W}$ implies that $W'\prec_p W$ we are also done in this case. Thus, the regularity of $L_G$ is established. To complete the proof of Proposition \ref{prop:LGisBiAutoStruct} we show that $L_G$ has the fellow-traveler property. We do so using Lemma \ref{lem:husdDist}. Suppose $W$ and $U$ are Peifer minimal and that there are $x,y\in\pmX\cup\Set{\eps}$ such that $xW=Uy$ in $G$. Let $M'$ be an $\Omega$-minimal derived diagram with boundary cycle $\rho\delta^{-1}$ such that $\rho$ labeled by $xW$ and $\delta$ is labeled by $Uy$. Choose a set of syllable-induced distinguished set of vertices and consider the proper $V(6)$ structure it induces on $M'$. By Proposition \ref{prop:CCtoReduce} and by the minimality of $W$ and $U$ we have that $M'$ is a $(\rho,\delta)$-thin diagram. Thus, the conditions of Lemma \ref{lem:husdDist} hold for $xW$ and $Uy$ with the constant $\kappa(G)$. Consequently, $xW$ and $Uy$ are $(2\kappa(G)+1)$-fellow-travelers. The proof of Proposition \ref{prop:LGisBiAutoStruct} is completed.

\section{Appendix}

In this appendix we prove Lemma \ref{lem:TechPropOfTwoGenGrp} and Lemma \ref{lem:ThreeSylbGeoSpecialCase} and also the Appel-Schupp syllable length condition (Definition \ref{def:ASSyllLenCond}) for finite Shephard groups on two generators. 

\begin{remark}
We conjecture that the syllable length condition (Definition \ref{def:ASSyllLenCond}) hold in general for all Shephard groups on two generators even without the finiteness assumption. We don't have at the moment proof of this assertion. However, here are some supporting evidences. As was already mentioned, it holds for Artin groups on two generators \cite[Lemma 6]{AS83}). Also, by using the result of this appendix, we can show that the following (non finite) two generator Shephard group $H$ has the syllable length condition. Let
\[
\Pres{a,b | a^{2r}=b^{2s}=1 \quad \br{a,b}{t} = \br{b,a}{t}}
\]
The following Coxeter group is a homomorphic image of $H$:
\[
\Pres{x,y | x^2=y^2=1 \quad \br{x,y}{t} = \br{y,x}{t}}
\]
But, the Coxeter group is a finite so the syllable length condition holds and consequently it holds in $H$. There are also other examples of similar nature.
\end{remark}

Let $\mathcal{S}(r,s,t)$ be a finite Shephard group generated by two generators with the following presentation:
\[
\Pres{a,b | a^r=1,\, b^s=1,\, \br{a,b}{t} = \br{b,a}{t}  }
\]

We start by verifying Lemma \ref{lem:TechPropOfTwoGenGrp} and the syllable length condition. We deal with Lemma \ref{lem:ThreeSylbGeoSpecialCase} at the end of this appendix. The classification of Todd and Shephard \cite{ST54} implies that the following condition holds when $t\geq3$:
\[
\frac{1}{r}+\frac{1}{s}+\frac{2}{t}>1
\]
Also, if $t$ is odd then $r=s$. Thus, we can give the possible values of $(r,s,t)$
and the sizes of $H(r,s,t)$ (computed by GAP); these are given in the following table:
\begin{center}
\begin{tabular}{|c|c|}
  \hline
  $(r,s,t)$ & Order of $H(r,s,t)$ \\
  \hline
  $(r,s,2)$ & $rs$ \\
  $(2,2,t)$ & $2t$ \\
  $(2,s,4)$ & $2s^2$ \\
  $(3,3,3)$ & $24$ \\
  $(3,3,4)$ & $72$ \\
  $(3,3,5)$ & $360$ \\
  $(4,4,3)$ & $96$ \\
  $(5,5,3)$ & $600$ \\
  $(3,5,4)$ & $1800$ \\
  \hline
\end{tabular}
\end{center}

The case where $t=2$ follows since then the group is the group $\mathbb{Z}_r\times \mathbb{Z}_s$. If $r=s=2$ then we have a finite Coxeter group on two generators and the theorem holds by Lemma 6 and Lemma 7 of \cite{AS83}. The last six cases in the table above can be verified using a computer (using software such as GAP or MAGMA). We suppress the gory details. However, to convince the reader, these are the finitely many group inequalities one has to check:

\begin{description}
\item For all $U\in\mathcal{W}(a,b)$ such that $\lls{a}{U}\leq r/2$ and $\lls{b}{U}\leq s/2$ we have:
\begin{enumerate}
 \item If $\lp{U}\leq4$ then $U\neq1$ (there are finitely many such
 words).
 \item If $t>3$ and $\lp{U}=6$ then $U\neq1$.
 \item If $t>4$ and $\lp{U}=8$ then $U\neq1$.
 \item If $\lp{U} = t$ and $|V|<|U|$ then $UV\neq1$
 \item If $\lp{U}<t$ and $|V|\leq|U|$ then $UV\neq1$
\end{enumerate}
\end{description}

We are left with the case where $r=2$, $t=4$, and $s$ is arbitrary. Take, $Q_1=\Pres{x,y|x^s=y^s=1,\, xy=yx}$, $Q_2=\Pres{z| z^2=1}$, and $\varphi:Q_1\to Q_1$ is the homomorphism such that $\varphi(x)=y$ and $\varphi(y)=x$. Then $\varphi^2=1$ so we have a semi-direct product $H=Q_1\underset{\varphi}{\ltimes} Q_2$ which has the following presentation:
\[
\Pres{x,y,z| \begin{array}{l}
  x^s=y^s=z^2=1,\, xy=yx,\, \\
  z x z = y,\, zyz = x
\end{array}
}
\]
$H$ is isomorphic to $G$ by the isomorphism $\psi:H\to G$ where
\[
\psi(x)=b,\; \psi(y)=aba,\; \psi(z)=a
\]
Hence,
\[
G \cong
(\mathbb{Z}_s\times\mathbb{Z}_s)\underset{\varphi}{\ltimes}\mathbb{Z}_2
\]
Knowing $G$ concretely allows us to verify the needed inequalities. Take $U\in\mathcal{W}(a,b)$. First, we show that if $U=1$ in $G$ non-trivially then $\lp{U}\geq8$. We can assume that $U$ is freely-reduced (because if $\widetilde{U}$ is obtained from $U$ by free reduction then $\widetilde{\lp{U}}\leq\lp{U}$). Let $k$ be the number of times $a$ appears in $U$. We claim that $k$ is even. Indeed, if $k$ is not even then $\psi^{-1}(U)$ contains $z$ and thus $\psi^{-1}(U)\neq1$ in $H$ which contradict our assumption that $U=1$ in $G$. Each element of $\mathcal{W}(a,b)$ has a cyclic conjugate that has even syllable length. Moreover, if $U=1$ in $G$ then each cyclic conjugate of $U$ equals $1$ in $G$. Therefore, we can assume that $\lp{U}$ is even and starts with $a$. This leaves us to check that $ab^{k_1}ab^{k_2}\neq1$ in $G$ where $0<|k_1|,|k_2|<s$. Indeed, $\psi^{-1}(ab^{k_1}ab^{k_2}) = zx^{k_1}zx^{k_2}=y^{k_1} x^{k_2} \neq 1$.
Next, suppose that $\lp{U} \leq 4$, $\lls{a}{U}=1$, $\lls{b}{U}\leq s/2$, and that $UV=1$ in $G$ non-trivially. We show that if $\lp{U} = 4$ then $|U|\leq|V|$. Suppose $UV = a b^{k_1} a b^{k_2} \cdots a b^{k_n}$ and $U = a b^{k_1} a b^{k_2'}$ where $k_2=k_2' + k_2''$. By above discussion we have that $n\geq 4$ and $n$ is even. Also, we have that $|U| = 2 + |k_1| + |k_2'|$. If we denote by $\ell$ the number of $a$ in $V$ then we have $|V|=\ell + |k_2''| + |k_3| + \cdots + |k_n|$. Now, $1=\psi^{-1}(UV) = y^{k_1 + k_3 + \cdots + k_{n-1}} x^{k_2' + k_2'' + k_4 + \cdots + k_{n}}$ and thus $s\mid k_1 + k_3 + \cdots + k_{n-1}$ and $s\mid k_2' + k_2'' + k_4 + k_6 \cdots + k_{n}$. So, there is some $t$ such that $ts-k_1 = k_3 + k_5 + \cdots + k_{n-1}$ which implies that $|k_3| + |k_5| + \cdots + |k_{n-1}| \geq |k_3 + k_5 + \cdots + k_{n-1}| = |tr-k_1| \geq |k_1|$ (using here the assumption that $|k_1|\leq s/2$). Similarly, $|k_2''| + |k_4| + \cdots + |k_{n-1}| \geq |k_2'|$. Consequently, $|V| \geq \ell + |k_1| + |k_2'|$. Since $n\geq4$ we get that $\ell\geq2$ and thus $|V| \geq 2 + |k_1| + |k_2'| = |U|$ as needed. Similar considerations prove this assertion if $UV = b^{k_1} a b^{k_2} \cdots a b^{k_n} a$. Also, a similar considerations prove that $|U|<|V|$ if $\lp{U} < 4$.

We conclude the appendix with the verification of Lemma \ref{lem:ThreeSylbGeoSpecialCase}. In the lemma $m_{ij}=3$ so this translates to $t=3$ in current notation. Thus $(r,s,t)$ is in the following set:
\[
\Set{(2,2,3),\, (3,3,3),\, (4,4,3),\, (5,5,3)}
\]
The first two cases are straightforward since when $r=s=2$ or $r=s=3$ the conditions $\lls{a}{U}\leq s/2 \leq 1.5$ and $\lls{b}{U}\leq s/2 \leq 1.5$, and the same conditions on $V$ imply that $\lp{U}=|U|$ and $\lp{V}=|V|$. Thus, if $U$ has the form $a^{\eps_1}b^{\eps_2}a^{\eps_3}$ then $V$ has the form $b^{\delta_1}a^{\delta_2}b^{\delta_3}$ ($\eps_1,\ldots,\delta_3 \in \Set{-1,1}$) and the lemma follows. We are left with the following cases: $r=s=4$ or $r=s=5$. Here, we may have $\lls{a}{U} = 2$ and $\lls{b}{U} = 2$. By exhaustive search in the group (which are finite; using a computer) we get that the only possible cases when $r=s=5$ are the following:
\begin{enumerate}
 \item $U=a b^2 a^{-2}$ and $V=b^{-1} a^{-1} b a^{-1} b$
 \item $U=a^2 b^{-2} a^{-1}$ and $V=b^{-1} a b^{-1} a b$
 \item $U=a^{-2} b^2 a$ and $V=b a^{-1} b a^{-1} b^{-1}$
 \item $U=a^{-1} b^{-2} a^2$ and $V=b a b^{-1} a b^{-1}$
\end{enumerate}
When $r=s=4$ we get the same cases with the exception that exponents $2$ can be replaced with $-2$ or vice versa. Thus, the conclusions of the lemma hold.

\end{document}